\documentclass[3p,times]{elsarticle}
\usepackage{titlesec}    
\usepackage[titletoc,toc,title]{appendix}  

\usepackage{amsthm} 
\usepackage{amsmath}
\usepackage{amsfonts}
\usepackage{amssymb}

\usepackage{tikz}
\usetikzlibrary{calc}

\usepackage{hyperref}  
\usepackage{bookmark} 

\def\makecoloredx b#1--#2{
        \expandafter\newcommand\csname b#1\endcsname{\ifmmode\else${}$\fi\bgroup\edef\startedat{\the\inputlineno}\color{#2}}%
        \expandafter\newcommand\csname e#1\endcsname{\egroup}%
        \expandafter\newcommand\csname b#1OK\endcsname{\bgroup\edef\startedat{\the\inputlineno}}%
        }
\def\Nmakecoloredx b#1--#2{
        \expandafter\newcommand\csname b#1\endcsname{}%
        \expandafter\newcommand\csname e#1\endcsname{}%
        \expandafter\newcommand\csname b#1OK\endcsname{}%
        }
\def\makecolored#1#2{\makecoloredx#1--{#2}} 
\let\optNNmakecolored\makecolored
{\makeatletter
\gdef\Pmakecoloredx b#1--#2{
        \expandafter\newcommand\csname b#1\endcsname{\message{ ERROR:  b#1 unexpected ! }{\endlinechar\m@ne \global\read\m@ne to\@gtempa}}%
        \expandafter\newcommand\csname e#1\endcsname{}%
        \expandafter\newcommand\csname b#1OK\endcsname{}%
        }
\makeatletter}
\gdef\NNmakecoloredx b#1--#2{
        \expandafter\newcommand\csname b#1\endcsname{\errmessage{ ERROR:  b#1 unexpected !}}%
        \expandafter\newcommand\csname e#1\endcsname{}%
        \expandafter\newcommand\csname b#1OK\endcsname{}%
        }
\def\optNNmakecolored#1#2{\NNmakecoloredx#1--{#2}} 

\optNNmakecolored{bledenWSxvi}{black!0!red!0!red}
\optNNmakecolored{bledenxvipoWS}{black!0!red!0!blue}
\optNNmakecolored{bunorpodruheschXVI}{black!0!red!0!red}
\optNNmakecolored{bunordruhapodruheschXVI}{black!0!red!0!blue}
\makecolored{brozdil}{black!0!blue}
\makecolored{bdubenxvi}{black!0!red}
\makecolored{bkvetenxvi}{black!0!red}
\makecolored{bzarixvi}{blue!100!red}

\providecommand\markTA{}
\providecommand\markTC{}
\providecommand\markTD{}

    \providecommand\bledenxv{}
    \providecommand\eledenxv{}

\providecommand\MKlistopadxvb{}
\providecommand\MKlistopadxve{}
\providecommand\blistopadxvH{}
\providecommand\elistopadxvH{}

\makeatletter
\newcommand\mynobreakpar{\par\nobreak\@afterheading}
\newcommand\myParBeforeItems{\text{}\mynobreakpar}
\makeatother

\providecommand\itemref[1]{(\ref{#1})}    
\theoremstyle{plain}
\newtheorem{theorem}{Theorem}[section]
\newtheorem{corollary}[theorem]{Corollary}
\newtheorem{lemma}[theorem]{Lemma}
\newtheorem{claim}[theorem]{Claim}
\newtheorem{proposition}[theorem]{Proposition}

\theoremstyle{definition}
\newtheorem{definition}[theorem]{Definition}
\newtheorem{question}[theorem]{Question}
\newtheorem*{ack*}{Acknowledgement}
\theoremstyle{remark}
\newtheorem{remark}[theorem]{Remark}
\newtheorem*{remark*}{Remark}

\newcommand\R{\mathbb R}

\newcommand\rn{\mathbb R^n}

\newcommand\N{\mathbb N}

\newcommand\C{\mathcal C}
\newcommand\F{\mathcal F}
\newcommand\eps{\varepsilon}

\iffalse

\newcommand\tsamark{\markYY{blue!40!white}}
\newcommand\tsbmark{\markYY{green!40!white}}
\else
\newcommand\tsamark{}
\newcommand\tsbmark{}
\fi

\def\tinyspacebeforewidehat{\tsbmark\hskip .0888em}  
\def\tinyspaceafterwidehat{\tsamark\hskip .1667em}   

\newcommand\theset{F}

\iffalse

\newcommand\ixsetmark{\markY{red!30!white}}
\newcommand\ixsetmarkAPPEN{\markY{green!30!white}}
\newcommand\opmark{\markY{blue!30!white}}
\else
\newcommand\ixsetmark{}
\newcommand\ixsetmarkAPPEN{}
\newcommand\opmark{}
\fi

\newcommand\ixsetGx{\ixsetmark \Gamma_{\! x}}
\newcommand\ixsetSx{\ixsetmark \mathcal J_{\!x}}
\newcommand\ixsetSxzero{\ixsetmark \mathcal J_{\!x_0}}
\newcommand\ixsetSbaru{\ixsetmark \mathcal J_{\!\bar u}}
\newcommand\ixsetSbarv{\ixsetmark \mathcal J_{\!\bar v}}
\newcommand\ixsetSxs{\ixsetmarkAPPEN \mathcal J_{\! x,s}}
\newcommand\ixsetSxsixm{\ixsetmarkAPPEN \mathcal J_{\! x,6^m}}
\newcommand\operSa{\opmark S_{\!\!a}}
\newcommand\operFa{\opmark F_{\!\!a}}

\newcommand\ddd{\varrho}

\DeclareMathOperator{\dist}{dist}
\DeclareMathOperator{\card}{Card}
\DeclareMathOperator{\spt}{supp}
\DeclareMathOperator{\Lip}{Lip}
\DeclareMathOperator{\Span}{span}
\DeclareMathOperator{\Tan}{Tan}
\DeclareMathOperator{\Ptg}{Ptg}

\newcommand\Fsigma{\text{F$_\sigma$}}

\DeclareMathOperator{\der}{der}

\newcommand\AAAA{\mathcal A}
\newcommand\BBBB{\mathcal B}

\DeclareMathOperator{\interior}{int}

\newcommand\fcolon{\colon}
\newcommand\setcolon{:}
\providecommand\boundary{\partial}
\newcommand{\closure}{\overline}

\titleformat{\section}{\bfseries}{\thesection.}{0.5em}{}
\titleformat{\subsection}{\normalfont\itshape}{\thesubsection.}{0.5em}{}
\titleformat{\subsubsection}{\normalfont\itshape}{\thesubsubsection.}{0.6em}{}

\begin{document}

\title{Extensions of vector-valued functions with~preservation of~derivatives\tnoteref{t1}\tnoteref{tack}}

\tnotetext[t1]{\relax
\
\\
{\bf
\copyright 2016. This manuscript version is made available under the CC-BY-NC-ND 4.0 license http://creativecommons.org/licenses/by-nc-nd/4.0/
}
}

\def\AckBody{
 The research leading to these results has received funding
 from the European Research Council / ERC Grant Agreement no. 291497.
 \\
 The first author would like to thank the private company
        RSJ a.s.
 for the support of his research activities.
 The second author was also supported by grants
 no. 14-07880S of GA\,\v{C}R
 and
 no. RVO 67985840 of Czech Academy of Sciences.
}

\tnotetext[tack]{\AckBody}


\author[am]{M.~Koc\corref{cor1}}
\ead{martin.koc@rsj.com}
\author[a1,a2]{Jan Kol\'a\v{r}}
\ead{kolar@math.cas.cz}


\address[am]{RSJ a.s., Na Florenci 2116/15, 110 00 Praha 1, Czech Republic}
\address[a1]{Institute of Mathematics, Czech Academy of Sciences, \v Zitn\'a 25, 115 67 Praha 1, Czech Republic}
\address[a2]{Mathematics Institute, University of Warwick, Coventry, UK \ (September 2014 -- October 2015)}
\cortext[cor1]{Corresponding author}

\begin{abstract}
  Let $X$ and $Y$ be Banach or normed linear spaces
  and $F\subset X$ a~closed set.
  We apply our recent extension theorem for vector-valued Baire one functions
  to obtain an extension theorem for vector-valued functions
  $f\colon F\to Y$  with pre-assigned derivatives,
  with preservation of differentiability
  (at every point where the pre-assigned derivative is actually a~derivative),
  preservation of continuity,
  preservation of (point-wise) Lipschitz property
  and
  (for finite dimensional domain $X$)
  preservation of strict differentiability
  and
  global (eventually local) Lipschitz continuity.
  This work depends on the paper
  \emph{Extensions of vector-valued Baire one functions with preservation of
  points of continuity}
  (M.~Koc \& J.~Kol\'a\v{r}, J.\,Math.\,Anal.\,Appl.~442:1).
\end{abstract}


\maketitle

  Our results can be roughly viewed as a~joint generalization of
  extension theorems of
  Tietze-Dugundji,
  Whitney ($C^1$-case) and
  McShane-Johnson-Lindenstrauss-Schechtman
  (see \cite[Theorem~2]{JLS}),
  all in point-wise fashion.
  Nevertheless, they were created as
  vector-valued
  generalizations
  of extension theorems of Aversa, Laczkovich, Preiss
  \cite{ALP}
  and Koc, Zaj\'\i\v{c}ek
  \cite{KZ}.

\section{Introduction}
\noindent
Differentiable extensions of functions were considered already in
the
1920's.
In \cite{J}, V.~Jarn\'ik proved that every real-valued differentiable function defined on a~perfect subset of~$\R$
can be extended to an everywhere differentiable function on~$\R$
(he even proved a~stronger form of this result with preservation
of Dini derivatives).
This result was independently obtained by G.~Petruska and M.~Laczkovich in \cite{PL}
(even with some additional estimates for the derivative of the extended function)
and generalized to real-valued differentiable functions defined on arbitrary closed subsets of~$\R$
by J.~Ma\v{r}\'ik in \cite{M}.
Extensions of vector-valued functions defined on (not necessarily closed) subsets of~$\R$
that preserve the derivative or even some other local properties
(e.g. boundedness, continuity or Lipschitz property) were investigated
by A.~Nekvinda and L.~Zaj\'i\v{c}ek in \cite{NZ}.

In \cite[Theorem 7]{ALP},
V.~Aversa, M.~Laczkovich and D.~Preiss
proved a~result concerning the extendibility
to a~real-valued differentiable function on $\rn$. In particular, they proved that
given a~function $f$ (defined on some nonempty closed set $F\subset\rn$)
and a~derivative of $f$ (with respect to $F$),
there exists an everywhere differentiable extension to $\rn$
that preserves the prescribed derivative if and only
if this prescribed derivative is a~Baire one function
on~$F$.
(Recall that a function is {\em Baire one} if
it is the point-wise limit of a sequence of continuous functions.)

The existence of continuously differentiable extensions of real-valued functions defined on closed subsets of $\rn$
was studied in \cite{W} by H.~Whitney already in
the
1930's (even with preservation of higher orders of smoothness).
The vector-valued case can be found, e.g., in \cite[Theorem 3.1.14]{Fed}.

In \cite{KZ},
M.~Koc and L.~Zaj\'i\v{c}ek
proved a~result that naturally jointly generalized both the extension
result
of V.~Aversa, M.~Laczkovich and D.~Preiss
\cite{ALP} as well as the $C^1$ case of the Whitney's extension theorem
for real-valued functions defined on closed subsets of $\rn$
(see, e.g., \cite[\S\,6.5]{EG}).
Their result
\cite[Theorem~3.1]{KZ}
can be roughly described as
a~theorem on extendibility
to a~differentiable function
with preservation of points of continuity of the derivative.
We were able to generalize this result further,
with the main
focus
on {\em vector-valued} functions.
\bledenWSxviOK
  We added several other new features, for example
  non-restrictive assumptions allowing arbitrary function (existing differentiability and continuity points are preserved),
  the preservation of point-wise, local and global Lipschitz property,
  or
  generalization to infinite-dimensional domains.
  One of the main contributions (extensions of vector-valued Baire one functions) was,
  due to its different nature and technical difficulty,
  moved to a~separate paper~\cite{KocKolarB1}.
\eledenWSxvi

Our main results on differentiable extensions
(see Theorem~\ref{thm:infinf} and Theorem~\ref{thm:fininf})
can be jointly formulated in the following way (recall that for
$p\in \N\cup\{\infty\}$,
$C^p$ denotes the class of
$p$-times continuously differentiable functions
in Fr\'echet sense;
note that the notion does not change if Fr\'echet sense is replaced by G\^ateaux sense):

\newcounter{saveenum}
\begin{theorem}
\label{thm:difext}
Let $X$, $Y$ be normed linear spaces,
$\theset\subset X$ a~closed set, $f\fcolon \theset\to Y$ an~arbitrary function
and $L\fcolon \theset\to\mathcal L(X,Y)$ a~Baire one function. 
Let $p\in\N\cup\{\infty\}$.
Then there exists a~function $\bar{f}\fcolon X\to Y$ such that
\begin{enumerate}[\textup\bgroup (i)\egroup]
   \item\label{thm:difext:item:ext}
   $\bar{f}=f$ on $\theset$,

   \item\label{thm:difext:item:cont}
   if $a\in \theset$ and $f$ is continuous at $a$ \textup(with respect to $\theset$\textup),
               then $\bar{f}$ is continuous at $a$,

   \item\label{thm:difext:item:hoelder}
   if $a\in \theset$, $\alpha\in (0,1]$ and $f$ is $\alpha$-H\"{o}lder continuous at $a$ \textup(with respect to $\theset$\textup),
                then $\bar{f}$ is $\alpha$-H\"{o}lder continuous at $a$;
   in~particular, if $f$ is Lipschitz at $a$ \textup(with respect to $\theset$\textup),
             then $\bar{f}$ is Lipschitz at $a$,

   \item\label{thm:difext:item:frechet}
   if $a\in \theset$ and $L(a)$ is a~relative Fr{\'e}chet derivative of $f$ at $a$
               \textup(with respect to $\theset$\textup), then $(\bar{f})^\prime(a)=L(a)$,

   \item\label{thm:difext:item:contcomp}
$\bar{f}$ is continuous on $X\setminus \theset$,

   \item\label{thm:difext:item:smoothcomp}
                if $X$ admits
                $C^p$-smooth partition of unity,
                then $\bar{f}\in{C^p}(X\setminus \theset,Y)$.

\setcounter{saveenum}{\value{enumi}}
\end{enumerate}
Moreover, if $\dim X<\infty$, then
\begin{enumerate}[\textup\bgroup (i)\egroup]
\setcounter{enumi}{\value{saveenum}}
   \item\label{thm:difext:item:strict}
               if $a\in \theset$, $L$ is continuous at $a$ and $L(a)$ is a~relative strict derivative
               of $f$ at $a$ \textup(with respect to $\theset$\textup), then the Fr{\'e}chet derivative
               $(\bar{f})^\prime$ is continuous at $a$ with respect to $(X\setminus \theset)\cup\{a\}$
               and $L(a)$ is the strict derivative of $\bar{f}$ at $a$ \textup(with respect to $X$\textup),
   \MKlistopadxvb
   \item\label{thm:difext:item:lip-Loc-Glob}
\bledenxvipoWSOK
               if $a\in \theset$,
               $R>0$,
               $L$ is bounded on
               $B(a,R) \cap \theset$
               and $f$ is Lipschitz continuous on
               $B(a,R) \cap \theset$,
               then $\bar{f}$ is Lipschitz continuous on
               $B(a,r)$ for every $r<R$;
               if
               $L$ is bounded on $\theset$
               and $f$ is Lipschitz continuous on
               $ \theset$,
               then $\bar{f}$ is Lipschitz continuous on
               $X$.
\eledenxvipoWS

   \MKlistopadxve
\end{enumerate}
\end{theorem}

\begin{remark}\label{rem:ZAthm:diffext}
\myParBeforeItems
\begin{enumerate}[(a)]
\item
                Statement \itemref{thm:difext:item:smoothcomp} of Theorem~\ref{thm:difext}
                can be modified as follows:
                if
                $X$ admits
                $\mathcal F$-partition of unity
                where $\mathcal F$ is a~fixed class of functions on $X$ from Lemma~\ref{l:XbezFjakoDefF},
                then
                $\bar{f}|_{X\setminus F}$ is of class $\mathcal F$
(where we say that $g|_U$ is of class $\mathcal F$,
on an open set $U$,
if for every $x\in U$, there is a~neighborhood $V$ of $x$ such that $g|_V$ is a~restriction of a~function from $\mathcal F$).
\item
The assumptions on partitions of unity cannot
be removed from (\ref{thm:difext:item:smoothcomp}),
see Proposition~\ref{prop:necessary}.
\item\label{rem:ZAthm:diffext:item:C1}
  In
  (\ref{thm:difext:item:strict}),
  if we additionally assume that
  $L(b)$ is a~relative Fr\'echet
  derivative of $f$ at $b$
  (with respect to $F$) for every $b\in U$
  where $U$ is a~relative neighborhood of $a$ in $F$,
  then
  $(\bar f)'$ is continuous at $a$ (with respect to $X$).
\item
  The condition $\dim X<\infty$ cannot be removed from \itemref{thm:difext:item:strict} or \itemref{thm:difext:item:lip-Loc-Glob},
  see Remark~\ref{rem:finDimNecess} for an example.
\end{enumerate}
\end{remark}
  If the Nagata dimension $\dim_N F$ of $F$ is finite
  (see \cite[Definition~4.26]{BBI} or \cite[p.~13]{BBII}),
  there is no obvious obstacle for extending with preservation of the Lipschitz property
  \cite[Theorem~6.26]{BBII}.\footnote{\relax
      \bunordruhapodruheschXVIOK
      Space $Y$ can be used as the range, see~\cite[Proposition~6.5]{BBII}
      \eunordruhapodruheschXVI
      whose proof works for normed linear spaces.}
  Thus we rise a~natural
  question:
\begin{question}
  Can the condition $\dim X< \infty$ in
  Theorem~\ref{thm:difext} (\itemref{thm:difext:item:strict} and \itemref{thm:difext:item:lip-Loc-Glob})
  be replaced by condition
  $\dim_N F < \infty$?
\end{question}

\begin{remark}
Let $F$ be a~closed subset of $\rn$, $Y$ be normed linear space.
The reader might wonder if,
for every function
$f \fcolon F \to Y$
differentiable at every point of $F$ (with respect to $F$),
there is a differentiable extension
$\bar f\fcolon \rn \to Y$.
The answer depends on quality of $F$ and kinds of differentiability under consideration:
\myParBeforeItems
\begin{enumerate}[a\textup)]
        \item
A condition on $F$
that assures the existence of a~differentiable extension for every differentiable function $f$ on $F$
can be found in \cite[Corollary~4.3]{KZ}.
It originates from \cite[Theorem~4~(ii)]{ALP};
in both papers, it was formulated with real-valued functions in mind only but it works in the vector-valued case too.
Indeed, for $F$ satisfying this condition, the relative derivative $L:=f^\prime$ is always a~Baire one
function on $F$ by \cite[Proposition~3~(ii), Theorem~4~(ii)]{ALP}
(their proofs work for vector-valued functions as well).
        For $F$ satisfying the condition,
the extension $\bar f$ can be obtained
by Theorem~\ref{thm:difext}.
        \item
        However,
        this extension does not necessarily exist
        for a~general set $F$:
        \cite[Theorem~5]{ALP} gives an example of a~compact set $F\subset \R^2$ and a~(uniquely) differentiable function $f$ on $F$
        such that $f'$ is not Baire one on $F$ and $f$ therefore cannot be extended to a~differentiable function on $\R^2$.
        \item
        A weaker condition on $F$ (namely that $\Span \Tan (F,x) = \rn$ for every $x\in \der F$)
        is sufficient if we ask for a~differentiable extension of a~{\em strictly} differentiable function (see \cite[Proposition~4.10]{KZ}).
        This result for real functions can be generalized to vector-valued functions.
        For more details,
        see Proposition~\ref{prop:KZ410},
        where the condition  $\Span \Tan (F,x) = \rn$ is relaxed to $\Span \Ptg (F,x) = \rn$.
        \item
        For a~positive result on $C^1$ extensions of strictly differentiable functions see
        \cite[Corollary~4.7]{KZ}
        which can be extended for vector-valued functions with the use of Theorem~\ref{thm:difext},
        otherwise following the proofs from~\cite{KZ}.
        Again, a~condition has to be imposed on the set $F$, see \cite[Example~4.14]{KZ}.
        \item
\bunordruhapodruheschXVIOK
        Proposition~\ref{prop:KZstrictstrict} contains another result 
        on
        $C^1$ extensions of strictly differentiable functions,
        with assumption of continuity of the derivative and, again,
        a~condition on the set $F$.
        \item
        For more on this topic see \cite{KolarClanekIII}.
\eunordruhapodruheschXVI
\end{enumerate}
\end{remark}

\begin{remark}
Note that all Hilbert spaces and spaces $c_0(\Gamma)$ with arbitrary set $\Gamma$
admit $C^\infty$-smooth partition of unity, all reflexive spaces and spaces with a~separable dual
admit $C^1$-smooth partition of unity, $L_p$ spaces with $p\in[1,\infty)$ admit partition
of unity of the same smoothness order as their canonical norms. The existence of a~$C^p$-smooth bump
implies the existence of $C^p$-smooth partitions of unity in all separable spaces
as well as in all reflexive spaces.

We refer the reader to Remark~\ref{rem:part} and Remark~\ref{rem:bumps} for more information
concerning the existence of partitions of unity in various spaces.
\end{remark}

\smallbreak 

The key ingredient in our proof of Theorem \ref{thm:difext} is our recently developed extension theorem
for vector-valued Baire one functions (cf. \cite[Theorem 6]{ALP} and \cite[Theorem 2.4]{KZ},
where special cases of this result were proved):


\begin{theorem}\cite[Theorem 1.1]{KocKolarB1}\label{thm:B1ext}
  Let $(X,\ddd)$ be a~metric space, $F\subset X$ a~closed set, $Z$ a~normed linear space
  and $L\fcolon F\to Z$ a~Baire one function. Then there
  exists a~continuous function $A\fcolon (X\setminus F) \to Z$ such that
  \begin{equation}\label{eq:ALP3}
  \lim_ {\substack{x\to a\\ x\in X\setminus F}}
  \left\|A(x) - L(a)\right\|_Z
  \frac{\dist (x,F)}{\ddd(x,a)} = 0
  \tag{NT}
  \end{equation}
  for every $a\in \boundary F$,
  \begin{equation}\label{eq:continuity}
  \lim_ {\substack{x\to a\\ x\in X\setminus F}} A(x) = L(a)
  \tag{C}
  \end{equation}
  whenever $a\in\boundary F$ and $L$ is continuous at $a$,
  and
\blistopadxvH
  \begin{equation}\label{eq:boundedness}
      \text{$A$ is bounded on $B(a,r) \setminus F$}
  \tag{B}
  \end{equation}
  whenever $a\in F$,
  $r\in (0,\infty)\cup\{\infty\}$
  and $L$ is bounded on $B(a, 12 r) \cap F$.
\elistopadxvH
\end{theorem}

Besides this introductory section, our paper consists of three more sections.
Section \ref{sec:prelim} contains the definitions of~the notions related to derivatives and partitions
of unity as well as an auxiliary proposition
about
relativizations
of partitions of unity to open subsets.
Section \ref{sec:infinf} is devoted to extensions of vector-valued functions
from closed subsets of {\em infinite} dimensional spaces, whereas Section \ref{sec:fininf}
extends the proofs of the previous section to obtain stronger results for {\em finite} dimensional domains.
\bunorpodruheschXVIOK
        There is also
        technical Appendix~\ref{apen:partition} on partitions of unity
        and
        Appendix~\ref{apen:KZ} containing a~small elaboration of a~theme from \cite[Section~4]{KZ}.
\eunorpodruheschXVI


\section{Basic notions and preliminaries}\label{sec:prelim}
\noindent
Let $X$ and $Y$ be normed linear spaces. We denote by $\mathcal L(X,Y)$
the set of all bounded linear operators from $X$ to $Y$.
For $u\in\mathcal L(X,Y)$, the number
$$\left\|u\right\|
_{\mathcal L(X,Y)}
:=\inf\left\{K>0\setcolon \left\|u(x)\right\|_Y\leq K\left\|x\right\|_X \text{ for every } x\in X\right\}$$
is called the \textit{norm} of the linear operator $u$.
The set $\mathcal L(X,Y)$ equipped with the norm
$\left\|\cdot\right\|_{\mathcal L(X,Y)}$
forms a~normed linear space, which is complete provided $Y$ is complete.

\begin{definition}\label{D:Rel_der}
Let $X$ and $Y$ be normed linear spaces, $A\subset X$ an arbitrary set,
$f\fcolon A\to Y$ a~function and $a\in A$.

\begin{itemize}
\item[(i)] A bounded linear operator $\operFa\fcolon X\to Y$ is called
           a~\textit{relative Fr{\'e}chet derivative of $f$ at $a$} (with respect to $A$)
           if either $a$ is an isolated point of $A$, or
    \begin{equation}
    \lim\limits_{\substack{x\to a\\x\in A}}\frac{\left\|f(x)-f(a)-\operFa(x-a)\right\|_Y}{\left\|x-a\right\|_X}=0.
    \end{equation}
\end{itemize}

\begin{itemize}
\item[(ii)] A bounded linear operator $\operSa\fcolon X\to Y$ is called
            a~\textit{relative strict derivative of $f$ at $a$} (with respect to $A$)
            if either $a$ is an~isolated point of $A$, or
    \begin{equation}\label{eq:strict-def}
    \lim\limits_{\substack{y\to a\\x\to a\\x,y\in A,\,x\neq y}}
    \frac{\left\|f(y)-f(x)-\operSa (y-x)\right\|_Y}{\left\|y-x\right\|_X}=0 \quad\text{(with $x=a$ or $y=a$ allowed)}.
    \end{equation}
\end{itemize}

\begin{itemize}
\item[(iii)] We say that $L\fcolon A\to\mathcal L(X,Y)$ is a~\textit{relative Fr{\'e}chet}
             (resp.\ \textit{strict}) \textit{derivative of $f$} (on $A$) if $L(a)$ is a~relative Fr{\'e}chet
             (resp.\ strict) derivative of $f$ at $a$ (with respect to $A$) for each $a\in A$.
\end{itemize}
\end{definition}

\begin{remark}
\myParBeforeItems
\begin{enumerate}[(a)]
\item At interior points of $A$, the notion of relative Fr{\'e}chet (resp.\ strict) derivative
      agrees with the classical notion of the Fr{\'e}chet (resp.\ strict) derivative.
      In particular, if it exists then it is uniquely determined.
\item Classically, Fr{\'e}chet differentiability of functions between normed linear spaces
      is introduced only for functions defined on open sets.
\item A relative strict derivative of $f$ is clearly also a~relative Fr{\'e}chet derivative of $f$.
\item If we consider $X=\rn$ and a~function $f\fcolon A\subset\rn\to Y$ in Definition \ref{D:Rel_der},
      then we can omit the assumption of boundedness for operators $\operFa$ and $\operSa $,
      since every linear function defined on a~finite-dimensional normed linear space
      is bounded automatically.
\item If $X=\rn$ then $\operFa$ is called a~{\em relative derivative of $f$ at $a$} or,
      if $a$ is an interior point of $A$, the~{\em derivative of $f$ at $a$}.
\end{enumerate}
\end{remark}

\begin{definition}\label{D:P-Lip}
Let $X$ and $Y$ be normed linear spaces, $A\subset X$ an arbitrary set,
$f\fcolon A\to Y$ a~function and $a\in A$.

If $\alpha\in (0,1]$, we say that \textit{$f$ is $\alpha$-H\"{o}lder continuous at $a$} (with respect to $A$)
if
either $a$ is an isolated point of $A$, or
$$\limsup\limits_{\substack{x\to a\\x\in A}}\frac{\left\|f(x)-f(a)\right\|_Y}{\left\|x-a\right\|^{\alpha}_X}<\infty.$$

We say that \textit{$f$ is Lipschitz at $a$} (with respect to $A$)
if it is $1$-H\"{o}lder continuous at $a$ (with respect to $A$), i.e.,
either $a$ is an isolated point of $A$, or
$$\limsup\limits_{\substack{x\to a\\x\in A}}\frac{\left\|f(x)-f(a)\right\|_Y}{\left\|x-a\right\|_X}<\infty.$$
%

As usual, $f$ is 
{\em point-wise $\alpha$-H\"older}
({\em point-wise Lipschitz})
if $f$ is
$\alpha$-H\"older
continuous
(Lipschitz)
at every $a\in A$.
And $f$ is
{\em locally $\alpha$-H\"older}
({\em locally Lipschitz})
if it is
$\alpha$-H\"older (Lipschitz)
(in the classical sense)
in a neighborhood of every $a\in A$.
\end{definition}


    If $U=X$ then the following definitions agree with the usual
    notions (see \cite[16.1., p.~165]{KM} or \cite[p.~304]{HHZ}).
    \bledenxv
    If $U\subset X$
    is a
    \eledenxv
    general open set,
    we essentially consider the restrictions to $U$.

\begin{definition}\label{def:rozklad}
  Let
  $X$ be a~metric space,
  $\F$ a~class of functions on $X$
  and
  $U\subset X$ an open set.

    A {\em locally finite partition of unity in $U$}
    (shortly a~{\em  partition of unity in $U$})
    is
    a~collection
    $\{\psi_\gamma\}_{\gamma\in\Gamma}$
    of real-valued functions $\psi_\gamma$ on
    $X$ such that
        \bledenxvipoWSOK
    $\sum_{\gamma\in\Gamma} \psi_\gamma(x) = 1$ for every $x\in U$ and
        \eledenxvipoWS
    there is a~neighborhood $V_y$ of $y$,
    for every $y\in U$,
    so that
    all but a~finite
    number of $\psi_\gamma$ vanish on $V_y$.

    If $\psi_\gamma \in \F$ for every $\gamma\in\Gamma$, we talk about an {\em $\F$-partition of unity}.
    If $\F$ is not specified, usually the continuous functions are assumed.

    We say that a~(locally finite) partition of unity
    $\{\psi_\gamma\}_{\gamma\in\Gamma}$
    in $U$
    is {\em subordinated to} an open cover $\mathcal U$ of $U$ if for every $\gamma\in\Gamma$
    there is $U_\gamma\in \mathcal U$ such that
    $\spt (\psi_\gamma) \subset U_\gamma$, where
    $\spt (\psi_\gamma) = \closure {\{x\in X\setcolon \psi_\gamma(x) \neq 0\}}$.

\bledenxv
\eledenxv
    We say that $U$
    {\em admits $\F$-partition of unity}
    if for every open cover $\mathcal U$ of $U$ there is a~locally finite
    $\F$-partition
    of unity $\{\psi_\gamma\}_{\gamma\in\Gamma}$ in $U$
    subordinated to $\mathcal U$.
\end{definition}

\begin{lemma}\label{l:XbezF}\label{l:XbezFjakoDefF}
   Let $X$ be a~normed linear space.
   Let $\F$ be
   \begin{enumerate}[\textup\bgroup (a)\egroup]
   \item
   the class of all continuous functions on $X$,
   or
   \item
   the class of all continuous functions on $X$ that are $p_1$-times
   G\^ateaux differentiable for some $p_1\in\N\cup\{\infty\}$,
   or
   \item
   the class of all $p_2$-times
   Fr\'echet differentiable functions
   on $X$ for some $p_2\in\N\cup\{\infty\}$,
   or
   \item
   the class of all
   $C^{p_3}$-smooth functions
   on $X$ for some $p_3\in\N\cup\{\infty\}$,
   or
   \item
   the class of all point-wise $\alpha_1$-H\"older continuous functions
   on $X$ for some $\alpha_1\in (0,1]$, or
   \item
   the class of all locally $\alpha_2$-H\"older continuous functions
   on $X$ for some $\alpha_2\in (0,1]$, in particular
   \\\vspace\itemsep
   the class of all
   locally Lipschitz continuous functions on $X$, or
   \item
   the intersection of two or several of the above classes \textup(for some $p_1$, $p_2$, $p_3$, $\alpha_1$, $\alpha_2$\textup).

   \end{enumerate}
   Let $\F^+$ be the class of all non-negative functions from $\F$.
   If $X$ admits $\F$-partition of unity,
   then every open set $U\subset X$ admits $\F^+$-partition of unity.
\end{lemma}

\begin{proof}
   We get immediately that $X$ admits $\F^+$-partition of unity.
   Indeed,
   given a~locally finite partition of unity $\{\psi_\gamma\}_{\gamma\in\Gamma}\subset \F$, we put
   $\widetilde \psi_\gamma = \psi_\gamma^2 / \sum_{\beta\in\Gamma} \psi_\beta^2$ for every $\gamma\in\Gamma$.
   Then $\sum_{\gamma\in\Gamma} \widetilde \psi_\gamma = 1$ and, for every $\gamma\in\Gamma$, $\widetilde \psi_\gamma \in \F^+$
   and $\spt \widetilde \psi_\gamma = \spt  \psi_\gamma$.

   Let $U\subset X$ be an arbitrary open set and
   let $\mathcal U$ be an open cover of $U$.
   Set $F=X\setminus U$.
   Define
   \[
     d_n =
     \begin{cases}
         1/n,        & n \in \N,
         \\
         \infty,     & n \le 0.
     \end{cases}
   \]

   Fix $n\in \N $. Let
\begin{align*}
    U_n &= \{ x \in X \setcolon d_n <   \dist(x, F) <   d_{n-3} \},
\\
    F_n &= \{ x \in X \setcolon d_{n-1} \le \dist(x, F) \le d_{n-2} \},
\\
    \mathcal U _ n  &=
    \{ U_n \cap G \setcolon G \in \mathcal U \} \cup \{ X \setminus F_n \}
    .
\end{align*}
    Then $\mathcal U_n$ is an open cover of $X$. 
    Let
    $\{\phi_{n,\gamma}\}_{\gamma \in \AAAA_n}$
    be an $\F^+$-partition of unity subordinated to $\mathcal U_n$.
    Let $\BBBB_n = \{ \gamma \in \AAAA_n \setcolon \spt \phi_{n,\gamma}
    \not\subset X\setminus F_n \}$.
    Then
    $\{\phi_{n,\gamma}\}_{\gamma \in \BBBB_n}$
    is subordinated to $\mathcal U$,
    $\sum_{\gamma \in \BBBB_n} \phi_{n,\gamma} (x) = 1$ for
    every $x \in F_n$
    and $\spt \phi_{n,\gamma} \subset U_n$ for every $\gamma \in \BBBB_n$.

    The family $\{ \phi_{n,\gamma} \setcolon n\in \N, \gamma \in \BBBB_n \}$
    is subordinated to $\mathcal U$ and locally finite in $U$.
    Every $x \in U$ belongs to one or two of the sets $F_n$,
    and at most three sets $U_n$.
    Thus
    \begin{equation}\label{eq:wdef}
        1\le w(x) := \sum_{ n\in \N, \gamma \in \BBBB_n }  \phi_{n,\gamma} \le 3
    \end{equation}
    for every $x \in U$.
    For every $n\in\N$ and $\gamma\in\BBBB_n$, let
    $\psi_{n,\gamma} (x) = \phi_{n,\gamma} (x)/ w(x)$
    for $x\in U$
    and note again that the sum in \eqref{eq:wdef} is finite in a~neighborhood of every point of
    $U\supset U_n \supset \spt \phi_{n,\gamma}$.
    For every $n\in\N$ and $\gamma\in\BBBB_n$, extend $\psi_{n,\gamma}$ by setting $\psi_{n,\gamma}(x) = 0$ for $x\in X\setminus U$.
    Then $\{ \psi_{n,\gamma} \} _ {n\in \N, \gamma \in \BBBB_n } $
    is a~locally finite $\F^+$-partition of unity in $U$.
\end{proof}

\begin{remark}\label{rem:part}
\myParBeforeItems
\begin{enumerate}[(a)]

   \item Every metric space admits partition of unity formed by continuous functions.
         Moreover, it even admits partition of unity formed by Lipschitz
         continuous
         (hence locally Lipschitz continuous, as used in Lemma~\ref{l:XbezFjakoDefF})
         functions
         (see \cite[the proof of Theorem]{Fried}).

   \item If $X$ is a~WCD
         Banach space, then $X$ admits partition of unity formed by continuous functions that are G\^ateaux differentiable
         \cite[Corollary~VIII.3.3]{DGZ}. The class of WCD spaces contains all separable and all reflexive Banach spaces
         \cite[Example~VI.2.2]{DGZ}.

   \item There is a~Banach space that is not WCD and admits $C^\infty$-smooth partition of unity,
         e.g.\ JL space of W.B.~Johnson and J.~Lindenstrauss
         (see \cite[p.~369]{DGZ}, for the definition of JL space, see \cite{JLspace}).

   \item If $X^*$ is a~WCG
         Banach space, then $X$ admits $C^1$-smooth partition of unity \cite[Corollary~VIII.3.11]{DGZ}.
         This includes all reflexive spaces as well as all spaces with a~separable dual.

   \item If a~Banach space $X$ admits a~LUR norm whose dual norm is also LUR, then $X$ admits $C^1$-smooth partition of unity
         \cite[Theorem~VIII.3.12~(i)]{DGZ}. Hence, if $K^{(\omega_1)}=\emptyset$, then $C(K)$ admits $C^1$-smooth partition of unity
         \cite[Corollary~VIII.3.13]{DGZ}. Note that if $K^{(\omega_0)}=\emptyset$, then $C(K)$ admits even $C^\infty$-smooth partition of unity
         (see \cite{DGZ-1990}).

   \item $L_p$ spaces with $p\in[1,\infty)$ admit partition of unity of the same smoothness order as their canonical norms, i.e. $C^\infty$-smooth partition
         of unity for $p$ even integer, $C^{p-1}$-smooth partition of unity for $p$ odd integer and, if $p$ is not an~integer, $C^{[p]}$-smooth partition
         of unity, where $[p]$ denotes the integer part of $p$ (see \cite[Corollary~VIII.3.11]{DGZ} and \cite[Theorem~V.1.1]{DGZ}).

   \item
   All Hilbert spaces
   and
   spaces $c_0(\Gamma)$ with arbitrary set $\Gamma$
   admit $C^\infty$-smooth partition of unity \cite[Theorem~2 and Theorem~3]{T} (see also \cite[16.16]{KM}).


\end{enumerate}
\end{remark}

\begin{remark}\label{rem:bumps} Let $p\in\N\cup\{\infty\}$.
\myParBeforeItems
\begin{enumerate}[(a)]

    \item
             It is
             an open problem whether every Banach space that admits
             a~$C^p$-smooth bump must also admit $C^p$-smooth partition of unity
             (see \cite[p.~370, Problem~VIII.1]{DGZ}, \cite[p.~179]{FM} and \cite[p.~172]{KM}).

    \item
   The existence of a~$C^p$-smooth bump implies the existence of $C^p$-smooth partitions of unity
   for example for separable spaces (see \cite{BF}
   or \cite[p.~360]{DGZ})
   and for reflexive spaces \cite[Theorem VIII.3.2]{DGZ}.

   More generally, it also holds
   for Banach spaces whose dual is WCG \cite[16.13(4)]{KM},
   for WCD Banach spaces \cite[p.~351, Theorem VIII.3.2]{DGZ} (cf.\ also \cite[53.15 and 16.18]{KM}),
             which includes reflexive spaces and separable spaces as we already noted,
   and for duals of Asplund spaces \cite[53.15 and 16.18]{KM}.


A result on Banach spaces with PRI and $C^p$-smooth partitions of
unity can be found in \cite[Corollary 4]{Haydon}. In particular,
Banach spaces with "nice" ("separable") PRI and with a~$C^p$-smooth
bump function admit $C^p$-smooth partition of unity,
see \cite[Remark 3.3]{GTWZ}, \cite[page 369, lines 26--27]{DGZ} and \cite[16.18]{KM}.

\end{enumerate}
\end{remark}



\section{Vector-valued functions in infinite dimensional domain}\label{sec:infinf}
\begin{theorem}\label{thm:infinf}
Let $X$, $Y$ be normed linear spaces, $F\subset X$ a~closed set,
$f\fcolon F\to Y$ an~arbitrary function
and $L\fcolon F\to\mathcal L(X,Y)$ a~function that is Baire one on $F$.
Then there exists a~function $\bar{f}\fcolon X\to Y$ such that
\begin{enumerate}[\textup\bgroup (i)\egroup]
   \item\label{thm:infinf:item:ext}
               $\bar{f}=f$ on $F$,

   \item\label{thm:infinf:item:cont}
               if $a\in F$ and $f$ is continuous at $a$ \textup(with respect to $F$\textup),
               then $\bar{f}$ is continuous at $a$,

   \item\label{thm:infinf:item:hoelder}
                if $a\in F$, $\alpha\in (0,1]$ and $f$ is $\alpha$-H\"{o}lder continuous at $a$ \textup(with respect to $F$\textup),
                then $\bar{f}$ is $\alpha$-H\"{o}lder continuous at $a$;
                in particular, if $f$ is Lipschitz at $a$
                \textup(with respect to $F$\textup), then $\bar{f}$ is Lipschitz at $a$,

   \item\label{thm:infinf:item:frechet}
                if $a\in F$ and $L(a)$ is a~relative Fr{\'e}chet derivative of $f$ at $a$
                \textup(with respect to $F$\textup), then $(\bar{f})^\prime(a)=L(a)$,

   \item\label{thm:infinf:item:contcomp}
                $\bar{f}$ is continuous on $X\setminus F$,

   \item\label{thm:infinf:item:smoothcomp}
                if
                $X$ admits
                $\mathcal F$-partition of unity
                where $\mathcal F$ is a~fixed class of functions on $X$ from Lemma~\ref{l:XbezFjakoDefF},
                then
                $\bar{f}|_{X\setminus F}$ is of class~$\mathcal F$.\footnote{\label{foot:restrclass}\relax
To provide a~formal definition for \itemref{thm:infinf:item:smoothcomp}, we say that $g|_U$ is of class $\mathcal F$
(on an open set $U$)
if for every $x\in U$, there is a~neighborhood $V$ of $x$ such that $g|_V$ is a~restriction of a~function from $\mathcal F$.}
\end{enumerate}
\end{theorem}

\begin{proof}
If $F=\emptyset$, the theorem trivially holds. Further suppose that $F$ is nonempty.
For every $x\in X$, we set
\begin{equation}\label{r(x)}
r(x):=\frac{1}{20} \dist(x,F).
\end{equation}
Further, for every $x\in X\setminus F$, we choose any point $\widehat x\in F$
such that
\begin{equation}\label{hat}
\left\|x-\widehat x\tinyspaceafterwidehat \right\|_X\leq 2\dist(x,F).
\end{equation}

If (\ref{thm:infinf:item:smoothcomp}) is under consideration, $X$ admits $\mathcal F$-partition of unity.
If this is not the case,
it admits at least continuous partition of unity
(since $X$ is a~metric space)
\bledenWSxviOK
and we let $\mathcal F$ be the class of continuous functions on $X$.
\eledenWSxvi

By Lemma~\ref{l:XbezF},
there exists a~non-negative locally finite
$\mathcal F$-partition of unity
$\{\phi_\gamma\}_{\gamma\in\Gamma}$ on $X\setminus F$ subordinated
to the covering $\{B(x,10r(x))\setcolon x\in X\setminus F\}$. So, in particular,
\begin{equation}\label{Pr0}
\{\phi_\gamma\}_{\gamma\in\Gamma}\subset{\mathcal F},
\end{equation}
\begin{equation}\label{Pr1}
0\leq\phi_\gamma
{\rm\ for\ every\ }\gamma\in\Gamma,
\end{equation}
\begin{equation}\label{Pr2}
\sum\limits_{\gamma\in\Gamma}\phi_\gamma(x)=1 {\rm\ for\ every\ } x\in X\setminus F
\end{equation}
and for every $\gamma\in\Gamma$ there is $x_\gamma\in X\setminus F$ such that
\begin{equation}\label{Pr3}
\spt\phi_\gamma\subset B(x_\gamma,10r(x_\gamma)).
\end{equation}

For every $x\in X\setminus F$, we denote
\begin{equation}\label{Sx1}
\ixsetGx :=\{\markTC \gamma\in \Gamma \setcolon B(x,10r(x))\cap B(x_\gamma,10r(x_\gamma))\neq\emptyset\}.
\end{equation}
Clearly, if $\markTD \gamma\in \markTD \Gamma \setminus \ixsetGx$ then $\phi_\gamma(x)=0$
by \eqref{Pr3}. Moreover, if $\markTA \gamma\in \ixsetGx$ \markTD then
\[
        \left|r(x)-r(x_\gamma)\right|
        \leq
        \Lip(r)  \left\|x-x_\gamma\right\|_X
        =
        \frac{1}{20}\left\|x-x_\gamma\right\|_X
        \overset{
                \eqref{Sx1}
        }{
        \leq
        }
        \frac{1}{20}\left(10r(x)+10r(x_\gamma)\right)
.
\]
\markTD
Hence
\begin{equation}\label{Pr4}
\frac{1}{3}\leq\frac{r(x)}{r(x_\gamma)}\leq 3\qquad {\rm{whenever}}\ \markTD \gamma\in \ixsetGx.
\end{equation}

\smallbreak
Let $A\fcolon (X\setminus F)\to\mathcal L(X,Y)$ be the function constructed
in Theorem \ref{thm:B1ext} (with
$Z=\mathcal L(X,Y)$).

Define $\bar f\fcolon X\to Y$ by
\begin{equation}\label{Ext_of_f}
\bar f(x):=
     \begin{cases}
\,
f(x) & \text{if $x\in F$,}\\
\,
\sum\limits_{\gamma\in\Gamma}\phi_\gamma(x)\left[f(\widehat{x_\gamma})+
A(x_\gamma)(x-\widehat{x_\gamma})\right]
& \text{if $x\in X\setminus F$}.
\end{cases}
\end{equation}

\smallbreak
Obviously,
$\bar{f}=f$ on $F$, which proves (\ref{thm:infinf:item:ext}).

Since linear mappings are $C^\infty$-smooth and the partition of unity
$\{\phi_\gamma\}_{\gamma\in\Gamma}$ is locally finite,
we easily conclude
using~\eqref{Pr0}
that $\bar f | _{X\setminus F}$ is of class $\mathcal F $.
        Assertions
(\ref{thm:infinf:item:smoothcomp}), if under consideration, and (\ref{thm:infinf:item:contcomp}) are therefore fulfilled.

\smallbreak
Let $a\in F$. For arbitrary $x\in X\setminus F$ and $\markTA \gamma\in \ixsetGx$,
by \eqref{r(x)}, \eqref{hat}, \eqref{Sx1} and \eqref{Pr4}, we get
\begin{equation}\label{E_1}
\left\|x_\gamma-x\right\|_X
\leq
10r(x_\gamma)+10r(x)
\leq
40r(x)
=
2\dist(x,F),
\end{equation}
 and likewise with $x_\gamma$ in the place of~$x$ on the right-hand side
\begin{equation}
\left\|x_\gamma-x\right\|_X\leq 10r(x_\gamma)+10r(x)\leq 40r(x_\gamma)
=
2\dist(x_\gamma,F),
\end{equation}
\begin{equation}
\left\| \tinyspacebeforewidehat \widehat x-x_\gamma\right\|_X\leq\left\| \tinyspacebeforewidehat \widehat x-x\right\|_X+\left\|x-x_\gamma\right\|_X
\leq 2\dist(x,F)+2\dist(x,F)=4\dist(x,F),
\end{equation}
$$
        \left\|\widehat{x_\gamma}-x_\gamma\right\|_X
    \leq
        2 \dist(x_\gamma, \theset)
    \leq
        2\left\|\tinyspacebeforewidehat \widehat x-x_\gamma\right\|_X
\leq 8\dist(x,F),$$
\begin{equation}\label{E_3}
\left\|\widehat{x_\gamma}-\widehat x\tinyspaceafterwidehat \right\|_X\leq\left\|\widehat{x_\gamma}-x_\gamma\right\|_X
+\left\|x_\gamma-\widehat x \tinyspaceafterwidehat \right\|_X
\leq 8\dist(x,F)+4\dist(x,F)=12\dist(x,F),
\end{equation}
\begin{equation}\label{E_4}
\left\|\widehat{x_\gamma}-x\right\|_X\leq\left\|\widehat{x_\gamma}-x_\gamma\right\|_X
+\left\|x_\gamma-x\right\|_X
\leq 8\dist(x,F)+2\dist(x,F)=10\dist(x,F),
\end{equation}
 and likewise
\begin{equation}\label{E_4.1}
\left\|\widehat{x_\gamma}-x\right\|_X\leq\left\|\widehat{x_\gamma}-x_\gamma\right\|_X
+\left\|x_\gamma-x\right\|_X
\leq 2\dist(x_\gamma,F)+2\dist(x_\gamma,F)=4\dist(x_\gamma,F).
\end{equation}
Since $\dist(x,F)\leq\left\|x-a\right\|_X$, by \eqref{hat}, \eqref{E_1} and \eqref{E_4},
we obtain
\begin{equation}\label{E_1_a}
\left\|x_\gamma-a\right\|_X\leq\left\|x_\gamma-x\right\|_X+\left\|x-a\right\|_X
\leq 3\left\|x-a\right\|_X,
\end{equation}
\begin{equation}\label{E_2_a}
\left\|\widehat{x_\gamma}-a\right\|_X\leq\left\|\widehat{x_\gamma}-x\right\|_X
+\left\|x-a\right\|_X\leq 11\left\|x-a\right\|_X,
\end{equation}
\begin{equation}\label{E_3_a}
\left\|\tinyspacebeforewidehat \widehat x-a\right\|_X\leq\left\|\tinyspacebeforewidehat \widehat x-x\right\|_X+\left\|x-a\right\|_X
\leq 3\left\|x-a\right\|_X.
\end{equation}

\smallbreak
Since
assertions
(\ref{thm:infinf:item:cont}), (\ref{thm:infinf:item:hoelder})
and (\ref{thm:infinf:item:frechet}) are clearly satisfied for $a\in\interior(F)$,
we will further assume that $a\in\boundary F$. If $x\in X\setminus F$,
by \eqref{Pr1}, \eqref{Pr2}, \eqref{Pr3}, \eqref{Sx1} and \eqref{Ext_of_f}, we obtain
\begin{align}
\left\|\bar f(x)-\bar f(a)\right\|_Y
&=\left\|\sum\limits_{\gamma\in\Gamma}\phi_\gamma(x)\left[f(\widehat{x_\gamma})
+A(x_\gamma)(x-\widehat{x_\gamma})-f(a)\right]\right\|_Y\nonumber\\
&=\left\|\sum\limits_{\gamma\in\Gamma}\phi_\gamma(x)\left[f(\widehat{x_\gamma})-f(a)
+L(a)(x-\widehat{x_\gamma})+(A(x_\gamma)-L(a))(x-\widehat{x_\gamma})\right]\right\|_Y\nonumber\\
&\leq\sum\limits_{\markTA \gamma\in \ixsetGx}\phi_\gamma(x)\left\|f(\widehat{x_\gamma})-f(a)\right\|_Y
+\sum\limits_{\markTA \gamma\in \ixsetGx}\phi_\gamma(x)\left\|L(a)\right\|_{\mathcal L(X,Y)}
\left\|x-\widehat{x_\gamma}\right\|_X\label{Est}\\
& \qquad+\sum\limits_{\markTA \gamma\in \ixsetGx}\phi_\gamma(x)\left\|A(x_\gamma)-L(a)\right\|_{\mathcal L(X,Y)}
\dist(x_\gamma,F)\frac{\left\|x-\widehat{x_\gamma}\right\|_X}{\dist(x_\gamma,F)}.\nonumber
\end{align}

\smallbreak
First suppose that $f$ is continuous at $a$ (with respect to $F$) and fix $\eps_1>0$.
There exists $\delta_1>0$ such that
\begin{equation}\label{Cont}
\left\|f(z)-f(a)\right\|_Y\leq\eps_1\qquad \text{for\ every\ }z\in F,
\ \left\|z-a\right\|_X<\delta_1.
\end{equation}
By \eqref{eq:ALP3} 
from Theorem~\ref{thm:B1ext}, there exists $\delta_2>0$ such that
\begin{equation}\label{Bd_beh1}
\left\|A(t)-L(a)\right\|_{\mathcal L(X,Y)}\dist(t,F)<\eps_1
\qquad \text{for\ every\ }t\in X\setminus F,\ \left\|t-a\right\|_X<\delta_2.
\end{equation}

Let $x\in X\setminus F$ be arbitrary with $\left\|x-a\right\|_X<\min\left(\eps_1,\frac{\delta_1}{11},\frac{\delta_2}{3}\right)$.
Then we deduce
from \eqref{E_1_a} and \eqref{E_2_a}
that
$\left\|x_\gamma-a\right\|_X<\delta_2$
and $\left\|\widehat{x_\gamma}-a\right\|_X<\delta_1$
for every $\markTA \gamma\in \ixsetGx$.
Thus by \eqref{Pr1}, \eqref{Pr2}, \eqref{E_4}, \eqref{E_4.1}, \eqref{Est},
\eqref{Cont}, \eqref{Bd_beh1} and $\dist(x,F)\leq\left\|x-a\right\|_X$,
we obtain
\begin{align}
\left\|\bar f(x)-\bar f(a)\right\|_Y&\leq\eps_1
+10\left\|L(a)\right\|_{\mathcal L(X,Y)}\eps_1+4\,\eps_1
=\left(5+10\left\|L(a)\right\|_{\mathcal L(X,Y)}\right)\eps_1.\nonumber
\end{align}
Since $\eps_1>0$ was arbitrary, $\bar f$ is continuous at $a$ and thus (\ref{thm:infinf:item:cont}) is proved.

\smallbreak
%

Next, suppose that $\alpha\in (0,1]$ and $f$ is $\alpha$-H\"older continuous at $a$ (with respect to $F$).
Then there exist $K>0$ and $\delta_3>0$ such that
\begin{equation}\label{Hoelder}
\left\|f(z)-f(a)\right\|_Y\leq K\left\|z-a\right\|^\alpha_X\qquad \text{for\ every\ }z\in F,
\ \left\|z-a\right\|_X<\delta_3.
\end{equation}
By \eqref{eq:ALP3} from Theorem~\ref{thm:B1ext}, there exists $\delta_4>0$ such that
\begin{equation}\label{Bd_beh2}
\left\|A(t)-L(a)\right\|_{\mathcal L(X,Y)}\dist(t,F)<\left\|t-a\right\|_X
\qquad \text{for\ every\ }t\in X\setminus F,\ \left\|t-a\right\|_X<\delta_4.
\end{equation}

Let $x\in X\setminus F$ such that $\left\|x-a\right\|_X<\min\left(\frac{\delta_3}{11},\frac{\delta_4}{3},1\right)$.
Then, for every $\markTA \gamma\in \ixsetGx$, using \eqref{E_1_a} and \eqref{E_2_a}
we get $\left\|x_\gamma-a\right\|_X<\delta_4$ and $\left\|\widehat{x_\gamma}-a\right\|_X<\delta_3$.
Similarly as above, by \eqref{Pr1}, \eqref{Pr2}, \eqref{E_4}, \eqref{E_4.1}, \eqref{E_1_a},
\eqref{E_2_a}, \eqref{Est}, \eqref{Hoelder}, \eqref{Bd_beh2} and $\dist(x,F)\leq\left\|x-a\right\|_X$,
we get
\begin{align}
\left\|\bar f(x)-\bar f(a)\right\|_Y&\leq K\sum\limits_{\markTA \gamma\in \ixsetGx}\phi_\gamma(x)\left\|\widehat{x_\gamma}-a\right\|^{\alpha}_X
+10\left\|L(a)\right\|_{\mathcal L(X,Y)}\left\|x-a\right\|_X\\
& \qquad +\ 4\sum\limits_{\markTA \gamma\in \ixsetGx}\phi_\gamma(x)\left\|x_\gamma-a\right\|_X\nonumber\\
&\leq\left(11^{\alpha}K+10\left\|L(a)\right\|_{\mathcal L(X,Y)}+12\right)\left\|x-a\right\|^{\alpha}_X,\nonumber
\end{align}
since $\left\|x-a\right\|_X\leq\left\|x-a\right\|^{\alpha}_X$ as $\left\|x-a\right\|_X<1$ and $\alpha\in (0,1]$.
Hence $\bar f$ is $\alpha$-H\"older continuous at~$a$ and (\ref{thm:infinf:item:hoelder}) is proved.

\smallbreak
Finally, we prove (\ref{thm:infinf:item:frechet}). Fix $\eps_2>0$.
Since $L(a)$ is a~Fr\'echet derivative of $f$ at $a$ (with respect to $F$),
there exists $\delta_5>0$ such that
\begin{equation}\label{Der}
\left\|f(z)-f(a)-L(a)(z-a)\right\|_Y\leq\eps_2\left\|z-a\right\|_X\qquad \text{for\ every\ }
z\in F,\ \left\|z-a\right\|_X<\delta_5.
\end{equation}
By \eqref{eq:ALP3} from Theorem \ref{thm:B1ext},
there exists $\delta_{6}>0$ such that
\begin{equation}\label{Bd_beh}
\left\|A(t)-L(a)\right\|_{\mathcal L(X,Y)}\frac{\dist(t,F)}{\left\|t-a\right\|_X}<\eps_2\qquad
\text{for\ every\ }t\in X\setminus F,\ \left\|t-a\right\|_X<\delta_{6}.
\end{equation}

Let $x\in X\setminus F$
be arbitrary satisfying
$\left\|x-a\right\|_X<\min\left(
\frac{\delta_5}{11}
,
\frac{\delta_{6}}{3}
\right)
$.
Then, for every $\markTA \gamma\in \ixsetGx$,
we get $\left\|x_\gamma-a\right\|_X<\delta_{6}$ and $\left\|\widehat{x_\gamma}-a\right\|_X<\delta_5$
by \eqref{E_1_a} and \eqref{E_2_a}.
Thus by \eqref{Pr1}, \eqref{Pr2}, \eqref{Pr3}, \eqref{Sx1},
\eqref{Ext_of_f}, \eqref{E_4.1}, \eqref{E_1_a}, \eqref{E_2_a},
\eqref{Der}
and
\eqref{Bd_beh},
we obtain
\begin{align*}
\left\|\bar f(x)-\bar f(a)-L(a)(x-a)\right\|_Y
   &
=\left\|\sum\limits_{\gamma\in\Gamma}\phi_\gamma(x)
\left[f(\widehat{x_\gamma})+A(x_\gamma)(x-\widehat{x_\gamma})-f(a)
-L(a)(x-a)\right]\right\|_Y\nonumber\\
\qquad\qquad&=\left\|\sum\limits_{\gamma\in\Gamma}
\phi_\gamma(x)\left[f(\widehat{x_\gamma})-f(a)-L(a)(\widehat{x_\gamma}-a)
+(A(x_\gamma)-L(a))(x-\widehat{x_\gamma})\right]\right\|_Y\nonumber\\
&\leq\sum\limits_{\markTA \gamma\in \ixsetGx}\phi_\gamma(x)
\left\|f(\widehat{x_\gamma})-f(a)-L(a)(\widehat{x_\gamma}-a)\right\|_Y\nonumber\\
&\qquad+\sum\limits_{\markTA \gamma\in \ixsetGx}\phi_\gamma(x)\left\|A(x_\gamma)-L(a)\right\|_{\mathcal L(X,Y)}
\left\|x-\widehat{x_\gamma}\right\|_X\nonumber\\
&\leq\sum\limits_{\markTA \gamma\in \ixsetGx}\phi_\gamma(x)\,\eps_2\left \|\widehat{x_\gamma}-a\right\|_X\nonumber\\
&\qquad+\sum\limits_{\markTA \gamma\in \ixsetGx}\phi_\gamma(x)\left\|A(x_\gamma)-L(a)\right\|_{\mathcal L(X,Y)}
\frac{\dist(x_\gamma,F)}{\left\|x_\gamma-a\right\|_X}\frac{\left\|x-\widehat{x_\gamma}\right\|_X}
{\dist(x_\gamma,F)}\left\|x_\gamma-a\right\|_X\nonumber\\
&\leq11\,\eps_2\left\|x-a\right\|_X+12\,\eps_2\left\|x-a\right\|_X=23\,\eps_2\left\|x-a\right\|_X.
\end{align*}
Since $\eps_2>0$ was arbitrary,
we finally get
$$\lim\limits_{\substack{
                        x\to a
                        \\
                        x\in X\setminus F
                        }}\frac{\left\|\bar f(x)-\bar f(a)-L(a)(x-a)\right\|_Y}{\left\|x-a\right\|_X}=0.$$
Since
$L(a)$ is a~Fr\'echet derivative of $\bar f$ at $a$ with respect to $F$,
we deduce
$(\bar{f})^{\prime}(a)=L(a)$, which proves (\ref{thm:infinf:item:frechet}).
\end{proof}

\bledenWSxviOK
By a~straightforward application of Theorem~\ref{thm:infinf}, we obtain the
following
generalization of \cite[Theorem~7]{ALP}
for infinite-dimensional domains and
vector-valued functions.
\eledenWSxvi

\begin{corollary}\label{ALP_do_Y}
Let $X$ be a~normed linear space that admits
Fr\'echet differentiable partition of unity,
$F\subset X$ a~nonempty closed set, $Y$~a~normed linear space, $f\fcolon F\to Y$
an arbitrary function and $L\fcolon F\to\mathcal L(X,Y)$ a~relative Fr\'echet derivative
of $f$ (with respect to $F$).
Then $L$ is Baire one on $F$ if and only if
there exists a~function
$\bar{f}\fcolon X\to Y$ such that $\bar{f}$ extends $f$, $\bar{f}$ is Fr\'echet differentiable
everywhere on $X$ and $(\bar{f})^{\prime}=L$ on~$F$.
\end{corollary}

The following proposition shows that the assumption on partitions of unity cannot be removed from (\ref{thm:infinf:item:smoothcomp}).
The remaining statements of Theorem~\ref{thm:infinf}
require only continuous partitions of unity which are available in all metric spaces.

\begin{proposition}\label{prop:necessary}
Let $X$ be a~normed linear space and $\widetilde{p}\in\N\cup\{\infty\}$. The following statements are equivalent:
\begin{enumerate}[\textup\bgroup (a)\egroup]
\item\label{prop:necessary:item:partition}
        The space $X$ admits $C^{\widetilde{p}}$-smooth partition of unity or partition of unity formed by continuous functions
$\widetilde{p}$-times differentiable in Fr{\'e}chet or G\^ateaux sense.
\MKlistopadxvb
\item\label{prop:necessary:item:vector-full}
        For every normed linear space $Y$, a~nonempty closed set $F\subset X$,
a~function $f\fcolon F\to Y$ and a~Baire one function $L\fcolon F\to\mathcal L(X,Y)$,
there exists a~function $\bar{f}\fcolon X\to Y$ that satisfies
the conclusions of Theorem~\ref{thm:infinf}
including
\blistopadxvH
   the respective conclusion of
   (\ref{thm:difext:item:smoothcomp}) with $p=\widetilde{p}$.
\elistopadxvH
%
\MKlistopadxve
\item\label{prop:necessary:item:scalar-minimal}
        Given any nonempty closed set $F\subset X$
and a
Fr\'echet smooth (or even
locally constant)
        function
        $f\fcolon F\to \R$,
there exists a~function $\bar{f}\fcolon X\to \R$ that satisfies
at least
the following properties from
Theorem~\ref{thm:infinf}:
   (\ref{thm:difext:item:ext}),
   (\ref{thm:difext:item:cont}) and
   the respective conclusion of
   (\ref{thm:difext:item:smoothcomp}) with $p=\widetilde{p}$.
\end{enumerate}
\end{proposition}

\begin{proof}
   The implication
   (\ref{prop:necessary:item:vector-full})
        $\Rightarrow$
        (\ref{prop:necessary:item:scalar-minimal})
                is obvious and
   (\ref{prop:necessary:item:partition})
        $\Rightarrow$
        (\ref{prop:necessary:item:vector-full})
                follows by
                    Theorem~\ref{thm:infinf}.
   The third implication
   (\ref{prop:necessary:item:scalar-minimal})
        $\Rightarrow$
        (\ref{prop:necessary:item:partition})
   follows by
   \cite[Lemma VIII.3.6, (ii) $\Rightarrow$ (i)]{DGZ}
   (or, more precisely,
   the proof of it, since
   the argument does not use
   the completeness
   of $X$)
   as soon as we show that
   given sets $A\subset W \subset X$, where $A$ is closed and $W$ open,
   there exists a~$C^{\widetilde{p}}$-smooth function
   $h \fcolon X \to [0,1]$ such that
   $A \subset  h^{-1} (0,\infty) \subset W$.
   To do so, assume $A$
   and
   $W$ are as indicated. Set $B=X\setminus W$ and $F=A\cup B$.
   Let $L(x)=0\in X^*$ for $x\in F$
   and
   $f(x)=1$ for $x \in A$, $f(x)=0$ for $x\in B$.
   By~(\ref{prop:necessary:item:scalar-minimal}),
   there exists
   a~function $\bar f$
   that satisfies
   conclusions
      (\ref{thm:infinf:item:ext}),
   (\ref{thm:infinf:item:cont})
   and (\ref{thm:infinf:item:smoothcomp})
   of Theorem~\ref{thm:infinf}
   (with $p=\widetilde{p}$).
    This extension $\bar f$
   is not necessarily
   $\widetilde{p}$-times continuously differentiable
   on the boundary of $F$. However, $h(x) := \varphi(\bar f(x))$ satisfies
   all required properties if $\varphi$ is a~suitable smooth function
   (e.g., $\varphi \fcolon \R \to [0,1]$ with $\varphi = 0$ on $(-\infty, 1/4]$
   and $\varphi=1$ on $[3/4, \infty)$);
   $h^\prime$ vanishes in a~neighborhood of the boundary of $F$
   by (\ref{thm:infinf:item:ext})
   and
   (\ref{thm:infinf:item:cont})).
\end{proof}


\section{Vector-valued functions in finite dimensional domain}\label{sec:fininf}
\noindent
\bledenxvipoWSOK
In this section, the domain space is
\eledenxvipoWS
the Euclidean space $\rn$ ($n\in\N$).
The norm on $\rn$ is denoted by $\left|\cdot\right|$. We identify $\rn$
with its dual space $(\rn)^*$ of all linear functionals on $\rn$.

It will be convenient to use the following {\em tensor product} notation.
If $\psi\in X^{*}$ and $y\in Y$, then $(y\otimes\psi)(u):=\psi(u)\,y$
for every $u\in X$. Note that $y\otimes\psi\in\mathcal L(X,Y)$.
In particular, if $\phi\fcolon \rn\to\R$ is differentiable at $x\in\rn$
and $y\in Y$, then $y\otimes{\phi^\prime(x)}\in\mathcal L(\rn,Y)$
and $(y\otimes{\phi^\prime(x)})(u)=(\phi^\prime(x))(u)\,y\in Y$ for every $u\in\rn$.
\bledenxvipoWSOK
Hence $y\otimes \phi^\prime(x)$
is the derivative
of
vector-valued function $t\mapsto \phi(t)\, y$
at $x$.
\eledenxvipoWS

The following theorem generalizes the main extension result from \cite{KZ}
to the case of vector-valued functions
\bledenxvipoWSOK
(see~Corollary~\ref{cor:KZ_do_Y},
\eledenxvipoWS
compare with \cite[Theorem~3.1]{KZ}).

\begin{theorem}\label{thm:fininf}
Let $F\subset\rn$ be a~closed set, $Y$~a~normed linear space,
$f\fcolon F\to Y$ an~arbitrary function and $L\fcolon F\to\mathcal L(\rn,Y)$
a~function that is Baire one on $F$. Then there exists a~function $\bar{f}\fcolon \rn\to Y$
such that
\begin{enumerate}[\textup\bgroup (i)\egroup]
   \item\label{thm:fininf:item:ext}
               $\bar{f}=f$ on $F$,

   \item\label{thm:fininf:item:cont}
               if $a\in F$ and $f$ is continuous at $a$ \textup(with respect to $F$\textup),
               then $\bar{f}$ is continuous at $a$,

   \item\label{thm:fininf:item:hoelder}
                if $a\in F$, $\alpha\in (0,1]$ and $f$ is $\alpha$-H\"{o}lder continuous at $a$ \textup(with respect to $F$\textup),
                then $\bar{f}$ is $\alpha$-H\"{o}lder continuous at $a$;
                in particular, if $f$ is Lipschitz at $a$
                \textup(with respect to $F$\textup), then $\bar{f}$ is Lipschitz at $a$,

   \item\label{thm:fininf:item:frechet}
               if $a\in F$ and $L(a)$ is a~relative Fr{\'e}chet derivative of $f$ at $a$
               \textup(with respect to $F$\textup), then $(\bar{f})^\prime(a)=L(a)$,

   \item\label{thm:fininf:item:infsmoothcomp}
               $
               \bar{f}
               | _ {\rn \setminus \theset}
               \in{C^\infty}(\rn\setminus F,Y)$,

   \item\label{thm:fininf:item:strict}
               if $a\in F$, $L$ is continuous at $a$ and $L(a)$ is a~relative strict derivative
               of $f$ at $a$ \textup(with respect to $F$\textup), then the Fr{\'e}chet derivative
               $(\bar{f})^\prime$ is continuous at $a$ with respect to $(\rn\setminus F)\cup\{a\}$
               and $L(a)$ is the strict derivative of $\bar{f}$ at $a$ \textup(with respect to $\rn$\textup),

   \item\label{thm:fininf:item:lip-Loc-Glob}
\bledenxvipoWSOK
               if $a\in F$,
               $R>0$,
               $L$ is bounded on
               $B(a,R) \cap F$
               and $f$ is Lipschitz continuous on
               $B(a,R) \cap F$,
               then $\bar{f}$ is Lipschitz continuous on
               $B(a,r)$ for every $r<R$;
               if
               $L$ is bounded on $F$
               and $f$ is Lipschitz continuous on
               $ F$,
               then $\bar{f}$ is Lipschitz continuous on
               $\rn$.
\eledenxvipoWS
\end{enumerate}
\end{theorem}

The strategy of the proof is analogous to the one used in the proof of Theorem \ref{thm:infinf}.
Assertions (\ref{thm:fininf:item:ext})-(\ref{thm:fininf:item:infsmoothcomp}) follow directly from Theorem~\ref{thm:infinf} as $\rn$ admits
$C^\infty$-smooth partition of unity. To ensure (\ref{thm:fininf:item:strict})-(\ref{thm:fininf:item:lip-Loc-Glob}),
we need a~special $C^\infty$-smooth
partition of unity in $\rn\setminus F$ that meets several additional requirements
analogous to those used in
proofs of \cite[Theorem~3.1]{KZ}
and the $C^1$ case of Whitney's extension theorem in \cite{EG},
namely \eqref{P1} and \eqref{P7} below.
Since we
decided to include
the preservation of the global Lipschitz continuity
(see~\itemref{thm:fininf:item:lip-Loc-Glob}),
we
had
to introduce a~slight change compared to \cite{KZ} and
\cite{EG}.

\begin{lemma}\label{l:specPart}
\bledenWSxviOK
There are 
$C_1, C_2 >1$
depending only on the dimension $n\in \N$
with the following property:
Let $F\subset\rn$ be a~nonempty closed set. There
\markTD
exist
$\{x_j\}_{j\in \N}\subset\rn\setminus F$
and
$\{\phi_j\}_{j\in\N}\subset C^{\infty}(\rn\setminus F,\R)$
\eledenWSxvi
such that,
letting
\begin{equation}
        \label{Sx}
        \ixsetSx :=\{\markTC j \in \N \setcolon B(x,10r(x))\cap B(x_j,10r(x_j))\neq\emptyset\}
\end{equation}
        and
\begin{equation}
        \label{eq:rbezmin}
        r(x):=\frac{1}{20}\dist(x,F),
\end{equation}
          we have,
for every $j\in\N$ and $x\in\rn\setminus F$,
\begin{equation}\label{P1}
        \card(\ixsetSx )
       \leq
        \blistopadxvH
        C_1,
        \elistopadxvH
\end{equation}
\begin{equation}\label{P2}
\frac{1}{3}\leq\frac{r(x)}{r(x_j)}\leq 3\qquad {\text{if}}\ \markTD j\in \ixsetSx ,
\end{equation}
\begin{equation}\label{P3}
0\leq\phi_j,
\end{equation}
\begin{equation}\label{P4}
\spt\phi_j\subset B(x_j,10r(x_j)),
\end{equation}
\begin{equation}\label{P5}
\sum\limits_{j\in \N}\phi_j(x)=1,
\end{equation}
\begin{equation}\label{P6}
\sum\limits_{j\in \N}{\phi_j}^{\prime}(x)=0
\end{equation}
and
\begin{equation}\label{P7}
                |{\phi_j}^{\prime}(x)|
        \leq
                \frac{
        \blistopadxvH
        C_2
        \elistopadxvH
                }{r(x)}.
\end{equation}
\end{lemma}

\blistopadxvH

The proof of Lemma~\ref{l:specPart} is standard.
It can be derived from a~very similar statement that is proven in \cite[pp.~245--247]{EG}
and summarized in \cite[Step~1 on p.~1031]{KZ}.
Statements in the same spirit can also be found in \cite{SingularIntegrals}
and \cite[Theorem~2.2]{MdeGuzman}.
For the sake of completeness,
we prove the lemma in
\bunordruhapodruheschXVIOK
Appendix~\ref{apen:partition}.  
\eunordruhapodruheschXVI

\begin{proof}[Proof of Theorem~\ref{thm:fininf}.]
If $F$ is empty, the theorem trivially holds. Further suppose that $F$ is nonempty.
Let $C_1, C_2>1$,
$\{x_j\}_{j\in \N}\subset\rn\setminus F$,
$\{\phi_j\}_{j\in \N}\subset C^{\infty}(\rn\setminus F,\,\R)$, $\ixsetSx $ and $r(x)$
be as in Lemma~\ref{l:specPart}. For every $x\in\rn\setminus F$, we choose any point
$\widehat x\in F$ such that
\begin{equation}\label{Dist}
\left|x-\widehat x\tinyspaceafterwidehat \right|=\dist(x,F).
\end{equation}

\smallbreak
Let $A\fcolon (\rn\setminus F)\to\mathcal L(\rn,Y)$ be the
function
constructed
in Theorem \ref{thm:B1ext} (with $X=\rn$ and $Z=\mathcal L(\rn,Y)$).
Define $\bar f\fcolon \rn\to Y$ by
\begin{equation}\label{Ext_of_f*}
\bar f(x):=
 \begin{cases}
\,
f(x)&\text{if $x\in F$},\\
\,
\sum\limits_{j\in \N}\phi_j(x)\left[f(\widehat{x_j})
+A(x_j)(x-\widehat{x_j})\right]
&\text{if $x\in\rn\setminus F$}.
 \end{cases}
\end{equation}


\smallbreak
As the formula for the extended function $\bar{f}$ is the same one
as in the proof of Theorem~\ref{thm:infinf} and the partition of unity
$\{\phi_j\}_{j\in \N}$ in $\rn\setminus F$ is only a~special case
of the partition of unity $\{\phi_\gamma\}_{\gamma\in \Gamma}$ used
in the proof of Theorem~\ref{thm:infinf},
assertions
(\ref{thm:fininf:item:ext})-(\ref{thm:fininf:item:infsmoothcomp}) follow immediately
by applying the proof of Theorem~\ref{thm:infinf} for the special case when $X=\rn$.

\medbreak
\MKlistopadxvb
It remains to prove
assertions
(\ref{thm:fininf:item:strict})-(\ref{thm:fininf:item:lip-Loc-Glob}).
%
%
We
\blistopadxvH
need
\elistopadxvH
some auxiliary estimates and computations.
\MKlistopadxve
Let $a\in F$. For arbitrary $x\in\rn\setminus F$ and $\markTA j\in \ixsetSx $,
by \eqref{P2}, \eqref{Sx}, \eqref{eq:rbezmin} and \eqref{Dist}, we get
\begin{equation}\label{Es_1}
|x_j-x|\leq 10r(x_j)+10r(x)\leq 40r(x)
=
2\dist(x,F),
\end{equation}
and likewise with $x_j$ in the place of~$x$ on the right-hand side
\begin{equation}
|x_j-x|\leq 10r(x_j)+10r(x)\leq 40r(x_j)
=
2\dist(x_j,F),
\end{equation}
\begin{equation}
|\tinyspacebeforewidehat \widehat x-x_j|\leq|\tinyspacebeforewidehat \widehat x-x|+|x-x_j|
\leq\dist(x,F)+2\dist(x,F)=3\dist(x,F),
\end{equation}
$$
        |\widehat{x_j}-x_j|
        = \dist(x_j, F)
        \leq|\tinyspacebeforewidehat \widehat x-x_j|\leq 3\dist(x,F),
$$
\begin{equation}\label{Es_3}
|\widehat{x_j}-\widehat x\tinyspaceafterwidehat |\leq|\widehat{x_j}-x_j|
+|x_j-\widehat x\tinyspaceafterwidehat |\leq 3\dist(x,F)+3\dist(x,F)=6\dist(x,F),
\end{equation}
\begin{equation}\label{Es_4}
|\widehat{x_j}-x|\leq|\widehat{x_j}-x_j|+|x_j-x|
\leq 3\dist(x,F)+2\dist(x,F)=5\dist(x,F)
.
\end{equation}
Since $\dist(x,F)\leq|x-a|$, by \eqref{Dist}, \eqref{Es_1} and \eqref{Es_4},
we obtain
\begin{equation}\label{Es_1_a}
|x_j-a|\leq|x_j-x|+|x-a|\leq 3|x-a|,
\end{equation}
\begin{equation}\label{Es_2_a}
|\widehat{x_j}-a|\leq|\widehat{x_j}-x|+|x-a|\leq 6|x-a|,
\end{equation}
\begin{equation}\label{Es_3_a}
|\tinyspacebeforewidehat \widehat x-a|\leq|\tinyspacebeforewidehat \widehat x-x|+|x-a|\leq 2|x-a|.
\end{equation}

For $x\in\rn\setminus F$, differentiating $\bar{f}$ at $x$, by \eqref{P1},
\eqref{P4}, \eqref{P5}, \eqref{P6},
\eqref{Sx} and \eqref{Ext_of_f*}, we get

\begin{align}
\nonumber
        (\bar{f}\,)^\prime(x)
&=
\sum\limits_{\markTA j\in \ixsetSx }
\phi_j(x)A(x_j)+\sum\limits_{\markTA j\in \ixsetSx }
\left[f(\widehat{x_j})+A(x_j)(x-\widehat{x_j})\right]
\otimes{\phi_j}^\prime(x)\\
&=\sum\limits_{\markTA j\in \ixsetSx }\phi_j(x)L(a)+
\sum\limits_{\markTA j\in \ixsetSx }\phi_j(x)\left[A(x_j)-L(a)\right]
\nonumber\\
&\qquad+\sum\limits_{\markTA j\in \ixsetSx }
\left[f(\tinyspacebeforewidehat \widehat x\tinyspaceafterwidehat )-L(a)(\tinyspacebeforewidehat \widehat x-x)\right]\otimes{\phi_j}^\prime(x)
\nonumber\\
&\qquad+\sum\limits_{\markTA j\in \ixsetSx }
\left[f(\widehat{x_j})-f(\tinyspacebeforewidehat \widehat x\tinyspaceafterwidehat )-L(a)
(\widehat{x_j}-\widehat x\tinyspaceafterwidehat )\right]\otimes{\phi_j}^\prime(x)
\nonumber\\
&\qquad+\sum\limits_{\markTA j\in \ixsetSx }
\left[(A(x_j)-L(a))(x-\widehat{x_j})\right]\otimes{\phi_j}^\prime(x)
\nonumber\\
&=L(a)+\sum\limits_{\markTA j\in \ixsetSx }\phi_j(x)\left[A(x_j)-L(a)\right]
\label{eq:Dfdole}
\\
&\qquad+\sum\limits_{\markTA j\in \ixsetSx }
\left[f(\widehat{x_j})-f(\tinyspacebeforewidehat \widehat x\tinyspaceafterwidehat )-L(a)
(\widehat{x_j}-\widehat x\tinyspaceafterwidehat )\right]\otimes{\phi_j}^\prime(x)
\nonumber\\
&\qquad+\sum\limits_{\markTA j\in \ixsetSx }
\left[(A(x_j)-L(a))(x-\widehat{x_j})\right]\otimes{\phi_j}^\prime(x).
\nonumber
\end{align}

\blistopadxvH

Now, we direct our attention to
assertions
\itemref{thm:fininf:item:strict}
and
\itemref{thm:fininf:item:lip-Loc-Glob}.
\begin{claim}\label{claim:strictAndLip}
Let
the following be defined as above:        
$F\subset \rn$, $Y$
a normed linear space,
$L\fcolon F\to \mathcal L(\rn,Y)$, $f\fcolon F\to Y$, $\bar f\fcolon \rn \to Y$ and
$A\fcolon ( \rn\setminus F ) \to \mathcal L(\rn,Y)$.
Suppose that
$a\in F$,
$r_1, r_2\in (0,\infty)\cup\{\infty\}$ and $K_1, K_2 \ge 0$ satisfy
\begin{align}\label{eq:AnearLa}
                \left\|A(t)-L(a)\right\|_{\mathcal L(\rn,Y)}
        &
        \le
                K_1
        &&
                \text{for\ every\ }t\in\rn\setminus F,\ |t-a|< r_1
\\\noalign{\noindent and}
\label{eq:LipOrStrict}
                \left\|f(z)-f(y)-L(a)(z-y)\right\|_Y
        &
        \leq
                K_2|z-y|
        &&
                \text{for\ every\ }
y,\,z\in F,\ \max\left(|y-a|,|z-a|\right)<r_2
.
\end{align}
For $x, y \in \rn$, denote
\begin{equation}\label{eq:Exy-def}
    E_{xy}:
  =
    \left\| \bar f(y) - \bar f(x) - L(a)(y-x) \right\|_Y
  =
    \sup_{\substack{T\in Y^*\\ \left\|T\right\|_{Y^*}\leq 1}}
    \left| T\left( \bar f(y) - \bar f(x) - L(a)(y-x) \right)\right|
    .
\end{equation}
Let
$r_3=\min\left( r_1/3, r_2/6 \right)$
and
$ K_3=
     (1+5\cdot 20 C_1 C_2) K_1
    +
     6\cdot 20 C_1 C_2 K_2
$,
where
$C_1$, $C_2$ are the constants from Lemma~\ref{l:specPart}.
Then
\begin{align}\label{eq:DfCont}
\left\|(\bar{f}\,)^\prime(x)-L(a)\right\|_{\mathcal L(\rn,Y)}
\le
        K_3
\qquad
&
\text{
                for all $x\in\rn\setminus F$ such that
                $
                        |x-a|
                        <
                        r_3
                $
          }
\\\noalign{\noindent and}
\label{eq:Exy-est}
   E_{xy}
\le
   33 K_3 \, \left | y - x \right |
\qquad
&
\text{
                for all $x,\,y \in\rn$
                such that
                $
                        \max\left(|x-a|,|y-a|\right)
                        <
                        r_3/2
                $.
          }
\end{align}
\end{claim}
\elistopadxvH

Postponing the proof of Claim~\ref{claim:strictAndLip},
we now proceed to the proof of
assertion
(\ref{thm:fininf:item:strict}). As its conclusion clearly holds
for $a\in\interior(F)$, we can further assume that $a\in\boundary F$.
\blistopadxvH
Fix $\eps_1>0$.
Note that $L$ is assumed to be continuous at $a$
(with respect to $F$).
By \eqref{eq:continuity} from Theorem \ref{thm:B1ext}, there exists
$ r_1 > 0 $ such that
(cf.~\eqref{eq:AnearLa})
\begin{equation}\label{Cont_of_A}
\left\|A(t)-L(a)\right\|_{\mathcal L(\rn,Y)}<\eps_1\qquad \text{for\ every\ }t\in\rn\setminus F,\ |t-a|< r_1.
\end{equation}
\elistopadxvH
Since we assume that $L(a)$ is a~strict derivative of $f$ at $a$ (with respect to $F$),
there exists $ r_2 > 0 $ such that
(cf.~\eqref{eq:LipOrStrict})
\begin{equation}
\left\|f(z)-f(y)-L(a)(z-y)\right\|_Y\leq\eps_1|z-y|\qquad \text{for\ every\ }
y,\,z\in F,\ \max\left(|y-a|,|z-a|\right)<r_2.
\end{equation}
By \eqref{eq:DfCont} from Claim~\ref{claim:strictAndLip} applied with $K_1=K_2=\varepsilon_1$,
we get $r_3>0$ such that
\begin{align}\label{eq:DfCont-applied}
\left\|(\bar{f}\,)^\prime(x)-L(a)\right\|_{\mathcal L(\rn,Y)}
&\le
        K_3
\qquad
&
&
\text{
                for all $x\in\rn\setminus F$ such that
                $
                        |x-a|
                        <
                        r_3
                $,
          }
\end{align}
with
$K_3=
 \left[1+220C_1C_2\right]\eps_1
$.
Since $\eps_1>0$ was arbitrary and $(\bar{f})^\prime(a)=L(a)$
(note that we already proved~\itemref{thm:fininf:item:frechet}),
we get that
$(\bar{f})^\prime$ is continuous at $a$ with respect to $(\rn\setminus F)\cup\{a\}$.

Likewise,
the estimate of $E_{xy}$
provided by
\eqref{eq:Exy-est}
shows
that $L(a)$ is the strict derivative of $\bar f$ at $a$.
Hence,
the proof of
assertion
(\ref{thm:fininf:item:strict})
is finished.

\smallbreak
\bledenxvipoWSOK
To prove
assertion
\itemref{thm:fininf:item:lip-Loc-Glob},
we prove that
               if $a\in F$,
  $r\in (0,\infty)\cup\{\infty\}$,
               $L$ is bounded on
               $B(a,72r) \cap F$
               and $f$ is Lipschitz continuous on
               $B(a,12r) \cap F$,
               then $\bar{f}$ is Lipschitz continuous on
               $B(a,r)$.
Both statements of~\itemref{thm:fininf:item:lip-Loc-Glob} then obviously follow
either
by a~standard compactness argument
or
using the case $r=\infty$.
\eledenxvipoWS

Assume that
               $a\in F$,
  $r\in (0,\infty)\cup\{\infty\}$,
               $L$ is bounded on
               $B(a,72r) \cap F$
               and $f$ is Lipschitz continuous on
               $B(a,12r) \cap F$.
               Let $K_0$ denote the Lipschitz constant of $f$.
               Then there is $C_0>0$ such that $\|A\|_{\mathcal L(\rn,Y)} \le C_0$ on
               $B(a,6r)\cap (\rn\setminus F)$ since $A$ was obtained from Theorem~\ref{thm:B1ext}
               (cf.~\eqref{eq:boundedness}).
             Thus we have~\eqref{eq:AnearLa} with $K_1= C_0 + \left\|L(a)\right\|_{\mathcal L(\rn,Y)}$ and $r_1=6 r$.
             Using the Lipschitz property of $f$, we obtain~\eqref{eq:LipOrStrict}
             with
             $K_2=K_0+\left\|L(a)\right\|_{\mathcal L(\rn,Y)}$
             and
             $r_2=12r$.
  An application of Claim~\ref{claim:strictAndLip}, namely of~\eqref{eq:Exy-est},
  gives
\[
  \left \| \bar f(y) - \bar f(x) \right\| _ Y
  \le \left( 33 K_3 + \left\| L(a) \right\|_{\mathcal L(\rn,Y)} 
      \right) \, \left | y - x \right|
\]
        for every $x,y \in \rn$ such that $\max( \left|x-a\right|, \left|y-a\right|)
        < r_3/2 = r$,
        which is the required Lipschitz property
        of~$\bar f$, cf.~\itemref{thm:fininf:item:lip-Loc-Glob}.

This concludes the proof of Theorem~\ref{thm:fininf}
except that we still have to show that Claim~\ref{claim:strictAndLip}
holds true.
\end{proof}

\begin{proof}[Proof of Claim~\ref{claim:strictAndLip}]
Let also the other symbols be defined as above
(that is, $x_j$, $\phi_j$ ($j\in \N$), $\ixsetSx$ and $r(x)$ ($x\in \rn\setminus \theset$)
are as in Lemma~\ref{l:specPart},
$\widehat x$ as in~\eqref{Dist} etc.).
Let
$x\in\rn\setminus F$ and
$
        |x-a|
        < r_3
        :=
        \min\left(\frac{r_1}{3},\frac{r_2}{6}\right)
$.
Then, for every $\markTA j\in \ixsetSx $, using \eqref{Es_1_a}, \eqref{Es_2_a}
and \eqref{Es_3_a}, we get $|x_j-a|<r_1$
and
$
\max (
        |\widehat{x_j}-a|,
        |\tinyspacebeforewidehat \widehat x-a|
       )
       <r_2
$.

\blistopadxvH
\bledenxvipoWSOK
By \eqref{eq:Dfdole},
\eledenxvipoWS
\begin{align}
\left\|
        (\bar{f}\,)^\prime(x)-L(a)
\right\|_{\mathcal L(\rn,Y)}
&\leq\sum\limits_{\markTA j\in \ixsetSx }\phi_j(x)\left\|A(x_j)-L(a)\right\|_{\mathcal L(\rn,Y)}
\nonumber\\
&\qquad+\sum\limits_{\markTA j\in \ixsetSx }
\left\|f(\widehat{x_j})-f(\tinyspacebeforewidehat \widehat x\tinyspaceafterwidehat )-L(a)(\widehat{x_j}-\widehat x\tinyspaceafterwidehat )\right\|_Y
\left|{\phi_j}^\prime(x)\right|\nonumber\\
&\qquad+\sum\limits_{\markTA j\in \ixsetSx }
\left\|A(x_j)-L(a)\right\|_{\mathcal L(\rn,Y)}|x-\widehat{x_j}|
\left|{\phi_j}^\prime(x)\right|\nonumber
.
\end{align}
        Estimating the first term
        by~\eqref{eq:AnearLa}
        together with~\eqref{P3} and \eqref{P5},
        the second one by~\eqref{eq:LipOrStrict}
        with~\eqref{P1},
        \eqref{P7}
        and
        \eqref{Es_3},
        and
        the third one by~\eqref{eq:AnearLa}
        with~\eqref{P1},
        \eqref{P7}
        and
        \eqref{Es_4},
        we get
\begin{align}
\left\|
        (\bar{f}\,)^\prime(x)-L(a)
\right\|_{\mathcal L(\rn,Y)}
&\leq
                K_1
        +
                6\dist(x,F)\,\frac{20 C_1 C_2}{\dist(x,F)}\, K_2
        +
                5\dist(x,F)\,\frac{20 C_1 C_2}{\dist(x,F)}\, K_1
                \nonumber
\\
   &\leq
     (1+5\cdot 20 C_1 C_2) K_1
    +
     6\cdot 20 C_1 C_2 K_2
 = K_3
.
\label{eq:DfContProved}
\end{align}
        Thus we obtained \eqref{eq:DfCont}.
\elistopadxvH

\medbreak

Next,
we want to prove~\eqref{eq:Exy-est}, the estimate of $E_{xy}$.
If (\ref{thm:fininf:item:cont})
of Theorem~\ref{thm:fininf}
were applicable at every point of $
a\in F
$, this could have been done easily using
the continuity of $(\bar f)'$ at
$
a\in F
$
(also the
mean value theorem
would be used on parts of the segment $L_{xy}$
together with the estimate
$E_{xy} \le  E_{xu} + E_{uv} + E_{vy}$
analogously to the arguments that follow),
but we can deal with the general case as well.

        From~\eqref{eq:AnearLa},
        we have
\begin{equation}\label{eq:Abounded}
          \left\|A(x)\right\|_{\mathcal{L}(\rn,Y)} \le M
\end{equation}
whenever $\left|x-a\right| < r_1$,
where
$ M = \left\| L(a) \right\|_{\mathcal{L}(\rn,Y)}
                        + K_1
$.

Note that $K_2  \le K_3$ by the definition of $K_3$, since $C_1, C_2 > 1$.
Fix $x,y\in\rn$ such that $\max\left(|x-a|,|y-a|\right)< r_3/2$.
We will show that
\[
   E_{xy} \le
   33
   K_3
   \left| y - x \right|
   .
\]
As this inequality trivially holds for $x=y$, we will further suppose that $x\neq y$.

\smallbreak

Let $L_{xy}$ denote the (closed) segment connecting $x$ and $y$.
We will distinguish several possible cases.

If $L_{xy} \subset\rn\setminus F$
and $T\in Y^*$ with $\|T\|_{Y^*}\leq 1$,
then there exists $\xi_T \in L_{xy}$ such
that
\[
        T\left( \bar f(y) - \bar f(x) \right)= \left((T(\bar f\,))'(\xi_T)\right) (y-x)
        .
\]
By \eqref{eq:DfContProved}, we simply get
\begin{equation}\label{eq:ExyVDoplnkuF}
        E_{xy}
    \le
        \sup_{\substack{T\in Y^*\\ \left\|T\right\|_{Y^*}\leq 1}} \left\| T\right\|_{Y^*}
        \left\| (\bar f\,)'(\xi_T) - L(a) \right\|_{\mathcal{L}(\rn,Y)}  \left| y - x \right|
    \leq
        K_3  \left| y - x \right|.
\end{equation}

If $x, y \in F$, we have $E_{xy} \leq K_2 \left| y - x \right|$ by \eqref{eq:LipOrStrict}.

In the remaining cases,
$L_{xy} \cap F \neq \emptyset$
and one or both points $x$, $y$
lie
in $\rn\setminus F$.

If $x,y\in\rn\setminus F$ then
segment $L_{xy}$ can be divided into two or three segments
as follows:

1.
$L_{xu}$ with $u\in F$ and $L_{xu} \setminus \{u\} \subset \rn\setminus F$.

2.
$L_{uv}$ with $u, v \in F$, which might possibly be degenerate ($v=u$).

3. $L_{vy}$ with $v\in F$ and $L_{vy} \setminus \{v\} \subset \rn\setminus F$.

\iftrue
  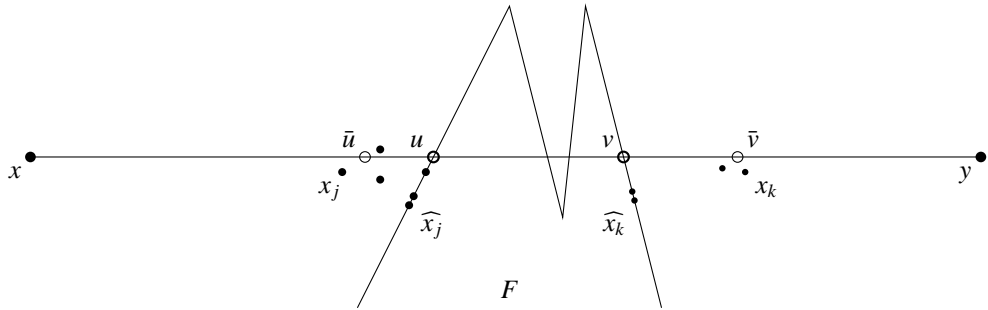
\begin{figure}[hbt]
  \begin{center}
\def\pX{(-4.3,0)}
\def\pY{(8.2,0)}
\begin{tikzpicture}
    \draw \pX node[below left]{$x$}
          --
          \pY node[below left]{$y$};
    \fill \pX circle (2pt);
    \fill \pY circle (2pt);
    \draw (0,-2) coordinate (a1)
        -- (2,2) coordinate (a2)
        -- (2.7, -0.8) --
           (3,2) coordinate (b1)
        -- (4,-2) coordinate (b2);
    \path (2, -2) node[above] {$F$};
    \draw[thick] (1,0) node[above left]{$u$}
          circle (2pt);
    \draw (0.1,0) node[above left]{$\bar u$}
          circle (2pt);
    \foreach \x/\y in {0.3/0.1, 0.3/-0.3, -0.2/-0.2
                       } {
          \fill (\x,\y) circle (1.5pt);
          \fill ($ (a1)!(\x,\y)!(a2) $) circle(1.5pt);
    }
    \path (0.7,-0.6) node [below right ]{$\widehat{x_j}$};
    \path (-0.1,-0.2) node [below left ]{$ x_j$};
    \draw[thick] (3.5,0) node[above left]{$v$}
          circle (2pt);
    \draw (5,0) node[above right]{$\bar v$}
          circle (2pt);
    \foreach \x/\y in {5.1/-0.2, 4.8/-0.15
                       } {
          \fill (\x,\y) circle (1.2pt);
          \fill ($ (b1)!(\x,\y)!(b2) $) circle(1.2pt);
    }
    \path (5.1,-0.2) node [below right ]{$x_k$};
    \path (3.65,-0.6) node [below left ]{$\widehat{x_k}$};
\end{tikzpicture}
    \caption{The case $x,y\in \mathbb R^n\setminus F$ with $L_{xy}$ intersecting $F$.
    Published with permission of \copyright\ Jan Kol\'a\v{r} 2016. All Rights Reserved.\relax
    }
    \label{fig:segmentLxy}
  \end{center}
  \end{figure}
\fi

\smallbreak

The reader might welcome an informal
remark, that we will
not use the estimate of $E_{uv}$
(which could be obtained immediately from \eqref{eq:LipOrStrict}),
but replace it by a~convex combination of estimates of $E_{\widehat{x_j}, \widehat{x_k}}$
with
$\widehat{x_j}$,
$\widehat{x_k}$
related to the definition
of $\bar f(\bar u)$, $\bar f(\bar v)$,
where $\bar u,\bar v \in \rn\setminus F$ are points approximating $u$, $v$
(see Figure~\ref{fig:segmentLxy}).
This way we do not need the continuity of $\bar f$ at points
$u,v\in F$.

We omit the case $x\in F$, $y\in \rn\setminus F$ since it is analogous to the case that follows.

If
\begin{equation}\label{eq:caseX}
  x\in \rn\setminus F \text{ and } y\in F
\end{equation}
then $L_{xy}$ divides into two segments
$L_{xu}$ and $L_{uv}$ as above with $v=y$
(we can consider $L_{vy}$ as degenerate).
We use a~convex combination of estimates of
$E_{\widehat{x_j}, y}$ (again provided by \eqref{eq:LipOrStrict}).
Apart from that, this case is similar to the
most complex
case $x,y\in \rn\setminus F$
and therefore we
will not fully threat both of them.
\bledenxvipoWSOK
  (Formally, the case \eqref{eq:caseX} can be treated together
  with the case $x,y\in \rn\setminus F$
\eledenxvipoWS
if we extend our notation as follows:
Let $\phi_0(z) = 1$ if $z\in F$ and $\phi_0(z) = 0$ if $z\in \rn\setminus F$.
Let $x_0= y$,
$\widehat{x_0}= y$
and
\bledenxvipoWSOK
$A(x_0) = 0$.
\eledenxvipoWS
Then
\bledenxvipoWSOK
 $\{ \phi_j \} _ {j\in \N\cup \{0\} }$
\eledenxvipoWS
is a~partition of unity
and  \eqref{Ext_of_f*} remains true, with unchanged values of $\bar f$,
if the sum is extended to include $j=0$.
Moreover, the second line of \eqref{Ext_of_f*}
then
gives the correct value of $\bar f(v)$
\bledenxvipoWSOK
even though we have $v=y\in F$.
\eledenxvipoWS
We also define $\ixsetSxzero =\markTD \{0\}$.)
\smallbreak

Let us concentrate on the case $x,y\in \rn\setminus F$.
Let
\begin{equation}\label{eq:defm}
 m:=
        \left| y-x \right| \min( K_3 / M , 1 / 4 )
        .
\end{equation}
  %
  %
  %
We choose a~point $\bar u \in L_{xu} \setminus \{u\}$
 with $\left| \bar u - u \right| < m$
and likewise $\bar v \in L_{vy} \setminus \{v\}$
 with $\left| \bar v - v \right| < m$.
(For the case \eqref{eq:caseX} we let $\bar v=v=y$.)
  %
  %
Since $L_{x\bar u} \subset \rn\setminus F$, we already estimated in~\eqref{eq:ExyVDoplnkuF} that
\begin{equation}\label{eq:Exu}
  E_{x\bar u}  \leq K_3  \left| \bar u - x \right|
 \le K_3 \left| y - x \right|
  .
\end{equation}
Likewise,
\begin{equation}\label{eq:Evy}
  E_{\bar vy}  \leq K_3  \left| y - \bar v \right|
 \le K_3 \left| y - x \right|
  .
\end{equation}
By \eqref{Es_4},
we have
\begin{align}
\label{eq:hatxj-baru}
        \left|\widehat{x_j} - \bar u\right|
\le
        5\dist (\bar u, F)
&
\le
        5 \left| \bar u - u \right|
< 5 m
 \\ \text{and}\quad
\label{eq:hatxk-barv}
        \left|\widehat{x_k} - \bar v\right|
&
\le
        5 \left| \bar v - v \right|
< 5 m
\end{align}
whenever
$\markTA j \in \ixsetSbaru$
and
$\markTA k \in \ixsetSbarv$, in which case therefore also
\begin{equation}\label{eq:hatxk-hatxj--barv-baru}
      \Bigl|
                \left(\widehat{x_k} - \widehat{x_j} \right)
                -
                \left( \bar v - \bar u \right)
      \Bigr|
     \le
        \left|\widehat{x_k} - \bar v\right| +
        \left|\widehat{x_j} - \bar u\right|
     \le
        10 m
.
\end{equation}
Since $\bar u \in L_{xy} \subset B(a,r_3/2)$, clearly
           $ \left| \bar u - a  \right | < r_3/2 $,
and from \eqref{Es_1}, we get
$
        \left| x_j - a \right|
\blistopadxvH
        \le \left| x_j - \bar u \right|
           + \left| \bar u - a     \vphantom{ x_j } \right |
        \le
            2 \dist(\bar u, F) +
            \left| \bar u - a  \right |
        \le
            3 \left| \bar u - a  \right |
        < 3 r_3/2
\elistopadxvH
        \le r_1
        $
whenever $\markTA j \in \ixsetSbaru$.
The values of $\bar f(\bar u)$ (and similarly also of $\bar f(\bar v)$) are defined by \eqref{Ext_of_f*}
where $\phi_j(\bar u)$ can be nonzero only when $\markTA j \in \ixsetSbaru$.
Using \eqref{Ext_of_f*}, the triangle inequality, \eqref{P3}, \eqref{P5}, \eqref{eq:Abounded}
and \eqref{eq:hatxj-baru}, we obtain
\[
   \Bigl\|
      \bar f(\bar u) - \sum_j \phi_j(\bar u) f(\widehat{x_j})
   \Bigr\|
    _Y
    \leq
    \sum_j \phi_j(\bar u) \left\|A(x_j)\right\|_{\mathcal{L}(\rn,Y)} \left| \bar u - \widehat{x_j}\right|
   \le
   5 M m
   \le
    5K_3 \left| y - x \right|
   .
\]
Likewise,
\[
   \Bigl\| \bar f(\bar v) - \sum_k \phi_k(\bar v) f(\widehat{x_k}) \Bigr\|
    _Y
   \le
   5 M m
   \le
    5K_3 \left| y - x \right|
    .
\]
Using identities $\phi_j=\phi_j\sum_k \phi_k$ and $\phi_k = \phi_k \sum_j \phi_j $, we can write
\begin{equation}\label{eq:esti1}
   \Bigl\|
        \bar f(\bar v) - \bar f (\bar u)
        -
        \sum_j \sum_{k} \phi_j(\bar u) \phi_k(\bar v) \left(  f(\widehat{x_k}) - f(\widehat{x_j}) \right)
   \Bigr\|
    _Y
   \le 10K_3 \left| y-x \right|
.
\end{equation}
Since $\widehat{x_j}, \widehat{x_k} \in F$,
we get
by \eqref{eq:LipOrStrict},
$K_2 \le K_3$,
\eqref{eq:hatxk-hatxj--barv-baru}
and
\eqref{eq:defm}
\[
         \left\| f(\widehat{x_k})
         -
         f(\widehat{x_j}) - L(a) (\widehat{x_k}- \widehat{x_j}) \right\|
    _Y
    \leq
         K_2 \left| \widehat{x_k} - \widehat{x_j} \right|
     \le
        K_3\,(
         10
            m
         + \left| \bar u - \bar v \right|
         )
     \le
         11
         K_3 \left| y - x \right|
\]
whenever $\markTA j \in \ixsetSbaru$ and $\markTA k \in \ixsetSbarv$.
Hence, again by \eqref{eq:hatxk-hatxj--barv-baru}, we obtain
(see also \eqref{eq:defm} and note that $\left\|L(a) \right\|_{\mathcal{L}(\rn,Y)}\le M$)
\[
         \left\| f(\widehat{x_k}) - f(\widehat{x_j}) - L(a) (\bar v - \bar u ) \right\|
    _Y
    \leq
     11
         K_3 \left| y - x \right| +
         10 m \left\|L(a) \right\|_{\mathcal{L}(\rn,Y)}
      \le
         21
          K_3 \left| y - x \right|
         ,
\]
which combines with \eqref{eq:esti1} and
$ \sum_j \! \sum_{k} \phi_j(\bar u) \phi_k(\bar v) = 1 $
to
\begin{equation}
   \Bigl\|
        \bar f(\bar v) - \bar f (\bar u)
        -
        L(a) (\bar v - \bar u )
   \Bigr\|
    _Y
   \le
         31
          K_3 \left| y - x \right|
 .
\end{equation}
So
\begin{equation}\label{eq:Euv}
            E_{\bar u \bar v}
   \le
         31
          K_3 \left| y - x \right|
 .
\end{equation}
By \eqref{eq:Exu}, \eqref{eq:Euv} and \eqref{eq:Evy},
since
$E_{xy} \le  E_{x \bar u} + E_{\bar u \bar v} + E_{\bar v y}$,
\begin{equation}\label{eq:strict-posledni-v-dukaze}
            E_{xy}
   \le
         33
          K_3 \left| y - x \right|
 ,
\end{equation}
which concludes the proof of Claim~\ref{claim:strictAndLip}.
\end{proof}


The following corollary provides a~vector-valued version of \cite[Theorem~3.1]{KZ}.
\begin{corollary}\label{cor:KZ_do_Y} 
Let $F\subset\rn$ be a~nonempty closed set, $Y$ a~normed linear space, $f\fcolon F\to Y$ an~arbitrary function
and $L\fcolon F\to\mathcal L(\rn,Y)$ a~relative Fr{\'e}chet derivative of $f$ (on $F$)
such that $L$ is Baire one on $F$. Then there exists a~function $\bar{f}\fcolon \rn\to Y$ such that
\begin{enumerate}[\textup\bgroup (i)\egroup]
   \item $\bar{f}$ is Fr{\'e}chet differentiable on $\rn$,
   \item ${\bar{f}}=f$ and $(\bar{f})^\prime=L$ on $F$,
   \item if $a\in F$, $L$ is continuous at $a$ and $L(a)$ is a~relative strict derivative of $f$ at $a$
(with respect to $F$), then the Fr{\'e}chet derivative $(\bar{f})^\prime$ is continuous at $a$,
   \item ${\bar{f}}\in\mathcal C^{\infty}\left(\rn\setminus F,Y\right)$.
\end{enumerate}
\end{corollary}

\begin{remark}
The previous corollary easily implies the $C^1$ case of
Whitney's extension theorem for vector-valued functions (see, e.g., \cite[Theorem~3.1.14]{Fed}).
Indeed, assuming that the assumptions of Whitney's theorem are fulfilled, it is sufficient
to show that $L(a)$ is a~strict derivative of $f$ at~$a$ for every $a\in F$
(which involves a~straightforward and easy computation only, cf.~\cite[Remark~3.2]{KZ})
and then to apply Corollary~\ref{cor:KZ_do_Y}.

\end{remark}

\let\pleaseDoSpellCheck(

\begin{remark}
\myParBeforeItems
\begin{enumerate}[(a)]

   \item In (\ref{thm:fininf:item:strict}) of Theorem \ref{thm:fininf}, we cannot expect the Fr{\'e}chet derivative
$(\bar{f})^\prime$ to be continuous at~$a$ with respect to the whole space $\rn$
unless appropriate assumptions are added (cf.\ Remark~\ref{rem:ZAthm:diffext}\itemref{rem:ZAthm:diffext:item:C1}).
Indeed,
consider $n=2$,
$F=[0,1]\times\{0\}$, $f\fcolon F\to\R$ given by $f(x,0)=x^7\left|\sin\frac{1}{x}\right|$ for $x\in(0,1]$
and $f(0,0)=0$, $L=0$ and $a=(0,0)$.
Note that $L(a)$ is a relative {\em strict} derivative of $f$ at $a$.
If we extend $f$ according to Theorem \ref{thm:fininf}, then the extended
function $\bar{f}$ is not Fr{\'e}chet differentiable in any neighborhood of $a$, since both $f$ and $\bar{f}$
do not have a~Fr{\'e}chet derivative at those points of~$F$ at which $\sin\frac{1}{x}$ changes its sign. However,
the Fr{\'e}chet derivative $(\bar{f})^\prime$ is continuous at $a$ with respect to $(\R^2\setminus F)\cup\{a\}$
as Theorem \ref{thm:fininf}\itemref{thm:fininf:item:strict} states.

\let\pleaseDoSpellCheck(

\smallbreak
\item\label{poznamka46itemB}
  In Theorem~\ref{thm:fininf}\itemref{thm:fininf:item:strict},
  neither the continuity of $(\bar f)'$ at $a$ nor the conclusion
  that $L(a)$ is the strict derivative of~$\bar f$ at~$a$
  (even with respect to $(\rn \setminus F) \cup \{a\}$)
  can be obtained
  when we remove the assumption
  that $L(a)$ is a~relative
  {\em strict}
  derivative of $f$ at $a$ (with respect to $F$).
Indeed, consider $F=\{0\}\cup\{\frac{1}{n}\setcolon n\in\N\}\subset \R $ and let $f\fcolon F\to\R$ be given by
$f(\frac{1}{n})=\frac{(-1)^n}{n^2}$ for $n\in\N$ and $f(0)=0$.
Let $L(x)=0$ ($x\in F$) and $a=0$.
\bledenWSxviOK
Obviously, $L(a)=0$ is a~relative derivative of $f$ at $a$ with respect to $F$.
\eledenWSxvi
For $n\in\N$, the distance between
$x_n := \frac{1}{n}$ and $x_{n+1}$
is less than
$\frac{1}{n^2}$ and the absolute increment of $f$ between these two points is greater than $\frac{1}{n^2}$.
Applying Theorem \ref{thm:fininf}, we obtain the extended function $\bar{f}$ on~$\R$ that is
continuous at
every $x_n$
due to condition
\itemref{thm:fininf:item:cont}.
By the mean value theorem,
for every $n\in \N$,
the absolute value of the derivative of $\bar{f}$ at some point
of the interval $(x_{n+1}, x_n)$
is greater than $1$.
Therefore
$(\bar{f})^\prime$ cannot be continuous at~$a$
with respect to $(\R\setminus F) \cup \{a\}$.
Also, $L(a)=0$ cannot be a~strict derivative of~$\bar f$ at~$a$
(even with respect to~$(\R\setminus F) \cup \{a\}$).

\let\pleaseDoSpellCheck(

\smallbreak
\item
  Likewise,
  neither of the two conclusions
  of Theorem~\ref{thm:fininf}\itemref{thm:fininf:item:strict}\footnote{\relax
          Namely
          that $(\bar f)'$ is continuous at $a$
          or 
          that $L(a)$ is the strict derivative of~$\bar f$ at~$a$
          (even
          only
          with respect to $(\rn \setminus F) \cup \{a\}$).
   }
  can be obtained without assuming that $L(a)$ is
  continuous at $a$
  with respect to $F$.
  Consider the same set $F$
  as in \itemref{poznamka46itemB} together with $a=0$, $f=0$ on $F$, $L(0)=0$ and $L(\frac 1n )=(-1)^n$ for $n\in \N$.
Applying Theorem \ref{thm:fininf}, we obtain the extended function $\bar{f}$ on~$\R$ that
has $(-1)^n$ as the
derivative at isolated point
$\frac 1n$
due to condition \itemref{thm:fininf:item:frechet}.
Hence $\bar{f}$ is continuous
at~$\frac 1n$.
By the mean value theorem,
for every $n\in \N$,
the absolute value of the derivative of~$\bar{f}$ at some point
close to $\frac 1n$
is greater than $\frac 12$.
Note that $(\bar f)'(0)=L(0)=0$.
Therefore
$(\bar{f})^\prime$ cannot be continuous at~$a$
with respect to $(\R\setminus F) \cup \{a\}$.
Also, $L(a)=0$ cannot be a~strict derivative of~$\bar f$ at~$a$
(even with respect to~$(\R\setminus F) \cup \{a\}$).
\end{enumerate}
\end{remark}

\begin{remark}\label{rem:finDimNecess}  
  If
  statements
  \itemref{thm:fininf:item:strict} or \itemref{thm:fininf:item:lip-Loc-Glob}
  are required
  in Theorem~\ref{thm:fininf},
  the
  assumption
  $F\subset \rn$ cannot be generalized to $F\subset X$, replacing
  $\rn$ by an (infinitely dimensional) Banach space
  $X$.
  In other words,
  the condition $\dim X<\infty$ cannot be removed from
  Theorem~\ref{thm:difext}
  \itemref{thm:difext:item:strict}\itemref{thm:difext:item:lip-Loc-Glob}.

  Indeed, let $p\in [1,2)$, $X=L_p(0,1)$, $Y=l_2$,
  $e\in X$ with $\left\|e\right\|_X =1$.
  By \cite[Theorem~3]{JLext},
  for every integer $n > 10$, there is a~finite set $F_n\subset X$
  and a~function $f_n \fcolon F_n \to Y$, such that
  $\Lip f > c_n \Lip f_n$ for every $f\fcolon X\to Y$ that extends $f_n$,
  where
  $c_n \to \infty$.
  The actual value of
  $c_n=\tau\cdot (\log n / \log \log n)^{1/p-1/2} $
  (for some $\tau > 0$)
  is not important for our purposes.
  By translating and scaling down the set, and by scaling the values of $f_n$
  we can assure that
  $F_n\subset B_X(2^{-n} e, 2^{-2n-1})$,
  $f_n(F_n) \subset B_{Y}(0, 2^{-2n-1})$
  and
  $\Lip f_n = c_n^{-1/2}$.
  Then the property of $f_n$ is that
  it has no extension $\bar f_n\fcolon X\to Y$ with
  $\Lip \bar f_n \le d_n := c_n \, \Lip f_n = c_n^{1/2}$.
  Note that $c_n^{-1/2} \to 0$ and $d_n\to \infty$.
  Let $F=\{0\} \cup \bigcup_{
  n
  > 10
  }
  F_n$ and define $f\fcolon F\to Y$ by $f(0)=0$, $f|_{F_n}=f_n$.
  The scaling was chosen so that $0\in \mathcal L(X,Y)$ is
  a~relative strict derivative of $f$ (with respect to~$F$) at $a:=0 \in X$.
  Since every $F_n$ is finite, $0$ is the only accumulation point of $F$.
  Let $L(x)=0$ for $x\in F$. Consider $\bar f$ that is an extension of $f$ as
  in the theorem.

  If $(\bar f)'$ is continuous at $a$
  with respect to $(X\setminus F) \cup \{ a \}$
  as in~\itemref{thm:fininf:item:strict},
  we obtain a~contradiction.
  First, we see that $(\bar f)'$ is actually continuous with respect to $X$
  (note that
  the derivative is
  continuous at $a
  =0
  $ with respect to $F$,
  since it equals $L$
  on~$F$ because every $x\in\theset \setminus\{a\}$ is an isolated
  point of $\theset$).
  Consequently,
  $\bar f$ is Lipschitz in $B_X(0,2r)$ for some $r>0$.
  Let $\pi$ be the
  radial
  projection onto $\closure {B_X(0,r)}$,
  $\pi(x) =  x $
  for $x\in B_X(0,r)$
  and
  $\pi(x) = r x / \left\| x \right\| $
  for $x\in X\setminus B_X(0,r)$.
  Then
        $\Lip \pi \le 2$, see e.g.\ \cite[Remark~4]{Maligranda},
        hence
  the mapping $g(x) = \bar f(\pi(x))$ is Lipschitz, and for $n$ sufficiently large,
  $g$ is a~$d_n$-Lipschitz extension of $f_n$, which is a~contradiction.
  Likewise, if $L(a)$ is the strict derivative of $\bar f$ at $a$
  (with respect to $X$)
      as in~\itemref{thm:fininf:item:strict}
  then
  $\bar f$ is Lipschitz in $B_X(0,2r)$ for some $r>0$
  and
      we obtain a~contradiction.
  The same example also provides a~contradiction with~\itemref{thm:fininf:item:lip-Loc-Glob}.
\end{remark}

\titleformat{\section}{\bfseries}{\appendixname~\thesection .}{0.5em}{}
\begin{appendices}
                                                             \def\warning{
                                                                         |
   WARNING: There is a faulty elsarticle command  \appendix,
   which we CANNOT use since the propositions then
   start by e.g. "Proposition Appendix A.1".
   JK                                                                    |
                                                                 \warning}
\section{Proof of Lemma~\ref{l:specPart}}\label{apen:partition}
We derive Lemma~\ref{l:specPart} from a~very similar statement that
is proven in \cite[pp.~245--247]{EG}
and summarized in \cite{KZ}.

\begin{lemma}[{cf.~\cite{KZ} and \cite{EG}}]\label{l:partKZ}
        Let $s\ge 1$.
        Then Lemma~\ref{l:specPart} holds true when \eqref{eq:rbezmin}
        is replaced by
\begin{equation}\label{eq:rwithMin}
  r(x) = \frac{1}{20} \min( s, \dist(x, F) )
  .
\end{equation}
\end{lemma}
\begin{proof}
        {\em Case $s=1$}.
        This case is exactly the one
        proven in \cite[pp.~245--247]{EG}
        and
        summarized in \cite[Step~1 on p.~1031]{KZ}.
        
        {\em Case $s > 1$}.
        The partition can be obtained from the previous case by scaling:
        \\
        Let $F_* = \{ x/s \setcolon x \in F \}$
        and let
        $
         \{
           x_{j}
         \}_{j\in\N}
        $,
        $
         \{
           \phi_{j}
         \}_{j\in \N}
        $
        be
        corresponding
        points and
        partition
        of unity from the previous case,
        that is,
        with
        the properties as in Lemma~\ref{l:partKZ}
        but
        with
        $s=1$
        and
        $F=F_*$.
For $j\in \N$ and $x\in \rn\setminus F$, let
 $\phi_{j}^*(x) = \phi_{j}(x/s)$
 and
 $x_j^* = s\,x_j$.
Then
partition of unity
$
 \{
 \phi_{j}^*
 \}_{j\in \N}
$
on $\rn\setminus F$
and points
$
 \{
 x_{j}^*
 \}_{j\in \N}
$
have all the required properties.
\end{proof}

\begin{proof}[Proof of Lemma~\ref{l:specPart}]
  We combine the partitions in such a~way that each of them is used in a~range of distances from $F$
  (with overlaps). We do that by multiplying each of them by a~function $v_{6^m}$ which is
  a~member of a~partition of unity that roughly depends only on distance from $F$.
  We can obtain
  $\{  v_{6^m} \}_{m\in \N}$
  either directly using the partitions at hand (as we do)
  or using the so called regularized distance.

  For $s\ge 1$, let $r_s$, $\{x_{s,j}\}_{j\in \N}$ and $\{\phi_{s,j}\}_{j\in \N}$ denote the
  function from
  \eqref{eq:rwithMin}
  corresponding to $s$,
  and the points and partition of unity
  from Lemma~\ref{l:partKZ}.
  Thus $r_s(x)=(1/20) \min(s, \dist(x,F))$. 
  Let also $r(x)=(1/20) \dist(x,F)$. 
  Note that $0\le \phi_{s,j} \le 1$ for every $s\ge 1$ and $j\in \N$.

For $d>0$, denote
$H_d = \{ x \setcolon \dist(x,F) \le d \}$.
Then, for every $x\in \rn \setminus F$,
\begin{equation}\label{eq:ballmezi}
        B(x,
                        10
        r(x))
        \subset
        H_{(3/2)\dist(x,F)}\setminus H_{(1/2)\dist(x,F)}
        =
        H_{30r(x)}\setminus H_{10r(x)}
        .
\end{equation}
Moreover, if $y\in B(x,
                        10
        r(x))$ then
\begin{equation}\label{eq:ballmezi2}
        B(y,
                        10
        r(y))
        \subset H_{(3/2)\dist(y,F)}\setminus H_{(1/2)\dist(y,F)}
        \subset H_{(9/4)\dist(x,F)}\setminus H_{(1/4)\dist(x,F)}
        =
        H_{45r(x)}\setminus H_{5r(x)}
        .
\end{equation}
Let
$J_s = \{ j \in \N \setcolon
                               x_{s,j}
                               \in H_{s/18}
                               \}
$
and $u_s=\sum_{ j \in J_s } \phi_{s,j}$ for all $s\ge 1$.
Also, let
\begin{align*}
        G_s &= \bigcup \left\{   B(x_{s,j}, 10r(x_{s,j})) \setcolon j \in J_s \right\}
         \subset
                H_{s/12}
        ,
        \\
        O_s &= \bigcup \left\{   B(x_{s,j}, 10r(x_{s,j})) \setcolon j \in \N \setminus J_s \right\}
         \subset
                \rn \setminus H_{s/36}
        .
\end{align*}
By \eqref{P4} and \eqref{eq:rwithMin}, we have
$\spt u_s
         \subset
         G_s
$
and
$\spt (1 - u_s) \subset
         O_s
$.
        In particular,
        $\spt u_{s/6} \subset H_{s/72} \subset H_{s/36}$
        does not intersect $\spt ( 1 - u_s)$ and
\begin{equation}\label{eq:uPrenasobene}
        u_s\, u_{s/6} = u_{s/6}
        \qquad
        \text{on }\rn\setminus F
        .
\end{equation}
%
Set (redefine) $u_1=0$ and
for $s\ge 6$, let
  \begin{figure}[bt]
  \begin{center}
\def\vyska{1.3}
\def\odstup{1}
\def\vyskaB{0.3}
\def\odstupB{-\odstup-\vyska-1}
\def\baselineB{ -\vyskaB \odstupB }
\def\krokB {-0.2}
\def\DRAWV #1,#2,#3;{
         \draw (2,-\vyska-\odstup) -- ++ (   #1   ,0) .. controls +(0.6,0) and +(-0.6,0) ..  ++(1,\vyska) \DRAWVrest#2,#3;
 }
\def\DRAWVrest#1,#2;{   -- node[below]{$v_{6^{  #2 }}$}    ( 2 + #1,-\odstup)
                               .. controls +(0.6,0) and +(-0.6,0) ..  ++(1,-\vyska) -- (16,-\vyska-\odstup);
 }
\noindent
\begin{tikzpicture}
 \draw (1,0) -- (16,0);
 \draw (2.5, 0) node[above] {$u_1$};
 \draw[thick,fill=gray] (1,0) -- ++(1,0) -- ++(0,\vyska) -- node[above]{$F$} +(-1,0) -- cycle;
 \draw (2,\vyska) -- ++ (    1   ,0 ) .. controls +(0.6,0) and +(-0.6,0) .. node[below left, very near start]{$u_{6^  {} }$}  ++(1,-\vyska) -- (16,0);
 \draw (2,\vyska) -- ++ (    4   ,0 ) .. controls +(0.6,0) and +(-0.6,0) .. node[below left, very near start]{$u_{6^  2  }$}  ++(1,-\vyska) -- (16,0);
 \draw (2,\vyska) -- ++ (    8   ,0 ) .. controls +(0.6,0) and +(-0.6,0) .. node[below left, very near start]{$u_{6^  3  }$}  ++(1,-\vyska) -- (16,0);
 \draw (2,\vyska) -- ++ (   13   ,0 ) .. controls +(0.6,0) and +(-0.6,0) .. node[below left, very near start]{$u_{6^  4  }$}  ++(1,-\vyska) -- (16,0);
 \draw[thick,fill=gray] (1,-\vyska-\odstup) -- ++(1,0) -- ++(0,\vyska) -- node[above]{$F$} +(-1,0) -- cycle;
 \draw (2,-\odstup) \DRAWVrest 1,{};
 \DRAWV 1,4,{2};
 \DRAWV 4,8,{3};
 \DRAWV 8,13,{4};
  \draw[thick,fill=gray] (1,\baselineB) -- ++(1,0) -- ++(0,\vyskaB) -- node[above]{$F$} +(-1,0) -- cycle;
 \draw (1,\baselineB) -- +(2,0) node[right] {$H_{1/6}$} ;
 \draw (1,\baselineB) ++(0,\krokB)  -- +(5,0) node[right] {$H_{1}$} ;
\end{tikzpicture}
    \caption{Schematic depiction of $u_s$, $v_s$ and $H_s$ (not to scale).
    Published with permission of \copyright\ Jan Kol\'a\v{r} 2016. All Rights Reserved.\relax
    }
    \label{fig:partition}
  \end{center}
  \end{figure}
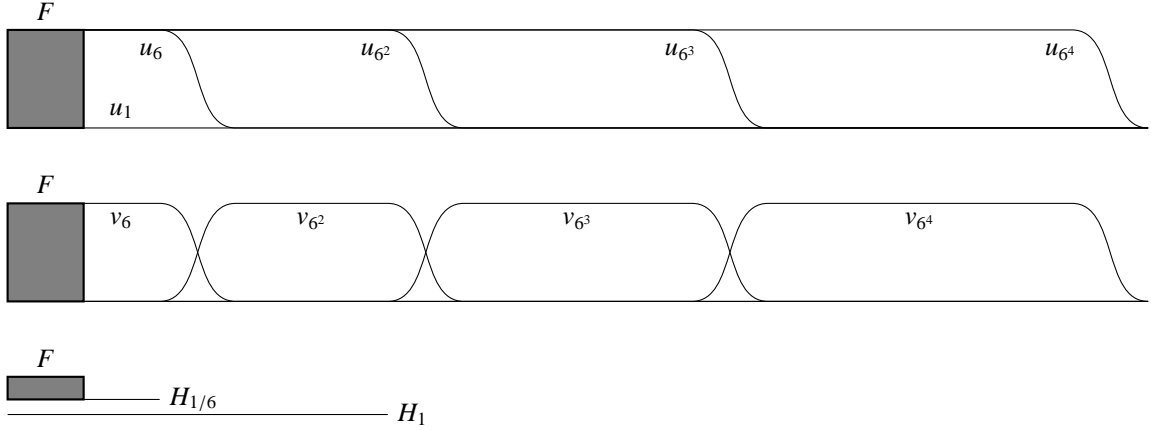
\[
   v_s(x) = u_s(x)\, (1-u_{s/6}(x)),
   \qquad x\in \rn\setminus F
   .
\]
Then
        $
        \spt v_s \subset G_s \cap O_{s/6}
        \subset H_{s/12}  \setminus H_{s/216}
        $
     for all $ s > 6 $
     and
        $
        \spt v_s \subset G_s 
        \subset H_{s/12} 
        $
     for $ s = 6 $.
   Also, $\sum_{m\in \N} v_{6^m}(x) = 1$ for every $x\in \rn \setminus F$
   (see Figure~\ref{fig:partition}).
   Indeed,
   on $\rn \setminus F$,
   \[
   1
   =
   \sum_{m\in \N} u_{6^m} - u_{6^{m-1}} 
   \overset
           {
           \eqref{eq:uPrenasobene}
           }
           {=}
   \sum_{m\in \N} u_{6^m} - u_{6^m}\, u_{6^{m-1}} 
   =
   \sum_{m\in \N} u_{6^m} \, ( 1 -  u_{6^{m-1}} )
   =
   \sum_{m\in \N} v_{6^m}
   .
   \]

  Let
  \begin{align*}
        M_s
        &=
        \{ j\in \N \setcolon
                             x_{s,j} \in  H_{s/6} \setminus H_{s/324}   \}
        \overset
                   {
                   \eqref{eq:ballmezi}
                   }
        \supset
        \{ j\in \N \setcolon
                             B(x_{s,j}, 10 r(x_{s,j})) \cap  (  H_{s/12} \setminus H_{s/216} ) \neq \emptyset   \}
           &
           \qquad
           &
           \text{for }s > 6,
   \\
        M_s
        &= 
        \{ j\in \N \setcolon
                             x_{s,j} \in  H_{s/6}   \}
        \overset
                   {
                   \eqref{eq:ballmezi}
                   }
        \supset
        \{ j\in \N \setcolon
                             B(x_{s,j}, 10 r(x_{s,j})) \cap     H_{s/12}  \neq \emptyset   \}
           &
           \qquad
           &
           \text{for }s = 6,
  \end{align*}
  and
  \[
        \ixsetSxs
         =
          \{
            j \in M_s \setcolon B(x,10r(x))\cap B(x_{s,j},10r(x_{s,j}))\neq\emptyset\}
        .
  \]
  If $\ixsetSxs \neq \emptyset$ then $x \in  H_{s/2} \setminus H_{s/972}$
  if $s>6$ and $x\in H_{s/2}$ if $s=6$.
  Thus, for a~fixed $x\in \rn\setminus F$, there are at most four different
  $s=6^m$
  such that
  $m\in \N $
  and
  $\ixsetSxsixm \neq\emptyset$.
  Moreover,
  $r(x)=r_s(x)$ for all $x\in H_s$,
  in particular
  $r( x_{s,j} ) = r_s( x_{s,j} )$
  for all $j\in M_s$,
  and
  $r( x ) = r_s( x )$
  whenever $\ixsetSxs\neq\emptyset$.
  Then also
  $\card \ixsetSxs \le C_1$
                by~\eqref{P1}.

  Consider
  \[
    \left\{ x_{6^m,j} \right\}
      _
       { m\in \N ,\ j \in M_{6^m} }
    \qquad
      \text{ and }
    \qquad
    \left\{ v_{6^m} (x)\, \phi_{6^m,j} (x) \right\}
      _
       { m\in \N ,\ j \in M_{6^m} }
       .
  \]
  Note that the condition $j \in M_{6^m}$ (compared to $j\in \N$) removes only (some of) the elements
  where $  v_{6^m} (x)\, \phi_{6^m,j} (x) $ is the zero function.
  Therefore
  $
  \sum _ {  m\in \N }
  \sum
       _ {  j \in M_{6^m} }  v_{6^m} (x)\, \phi_{6^m,j} (x)
  =
  \sum _ {  m\in \N  } \sum _ {j \in \N }  v_{6^m} (x)\, \phi_{6^m,j} (x)
  =
  \sum _ {  m\in \N  }  v_{6^m} (x)
  = 1
  $
  for $x\in \rn\setminus F$.
  Obviously,
  $\spt  v_{6^m} \, \phi_{6^m,j} \subset \spt  \phi_{6^m,j}  \subset B(x_{6^m,j}, 10 r_s(x_{6^m,j})) = B( x_{6^m,j}, 10 r(x_{6^m,j}))$
  whenever $j \in M_{6^m}$.

Let $\eta$ be a
bijection $\eta\fcolon \N \to \{ (6^m, j) \setcolon m\in \N, \ j \in M_{6^m} \} $.
Let
$ \phi_k^\# (x) = v_s (x) \, \phi_{s,j} (x)$
and
$ x_k^\# (x) = x_{s,j} $
for
$x\in \rn\setminus F$,
where $(s,j)=\eta(k)$.

We claim that
$\{ \phi_k^\# \}_{k\in \N}$
is a~partition of unity
in $\rn\setminus F$
with the required properties
but with $C_1$, $C_2$ replaced by
$C_1^*:= 4 C_1$,
$C_2^* := 3 C_1 C_2$.
        To show that,
        fix $x\in \rn\setminus F$.
        Since $r_s(x) \le r(x)$ for every $s\ge 1$,
        we have
        \[
        \ixsetSx \subset \{ k \in \N \setcolon j \in \ixsetSxs \text{ where } (s,j) = \eta(k) \}
        .
        \]
        Recall that
        there are at most four different $s=6^m$ ($m\in \N$) such that $\ixsetSxs\neq \emptyset$.
        And, for each such $s$, $\card \ixsetSxs \le C_1$. Hence $\card \ixsetSx  \le 4 C_1$.
        Furthermore,
        for every $k \in \N$
        and
        $(s,j) =\eta(k)$,
\begin{equation*}
        \left| (\phi_k^\#)' (x) \right|
 \le
             \left|
                   \phi'_{s,j}(x)
             \right|
             +
             \left|
                   u'_{s} (x)
             \right|
             +
             \left|
                    u'_{s/6} (x)
             \right|
     \le
      C_2/r(x) + 2C_1C_2/r(x)
     \le
      3 C_1 C_2 / r(x)
                   .
                   \qedhere
\end{equation*}
\end{proof}
\elistopadxvH

\section{Extensions  from special closed sets  $F\subset \rn$}
\label{apen:KZ}
\noindent
Forthcoming paper \cite{KolarClanekIII}
contains generalizations of
results of
\cite[Section 4]{KZ}
that avoid the assumption of $L$ being Baire one or continuous
and replace them by requirements on the set $\theset$.
We give here some of the results with concise and self-contained proofs.
For a stronger theorem with
lengthy proof
and a number of other corollaries, see \cite{KolarClanekIII}.

        Recall from \cite{KZ} that,
        for $\theset\subset \rn$ and $x\in \rn$,
\begin{equation}\label{eq:TanPtg}
        \Tan(\theset,x) \subset \Ptg(\theset,x)
        ,
\end{equation}
        where
        \begin{align*}
                \Tan(\theset,x) &= \{ v \in \rn \setcolon \text{there
                are
                      $x_k \in \theset$
                and $\alpha_k\in [0,\infty)$
                (for all $k\in \N$)
                such that
                $x_k\to x$
                and
                $\alpha_k\,(x_k-x) \to v$} \}
                \\
                \noalign{\noindent and}
                \Ptg(\theset,x) &= \{ v \in \rn \setcolon \text{there
                are
                      $x_k, y_k \in \theset$
                and $\alpha_k\in \R$
                (for all $k\in \N$)
                such that
                $x_k\to x$, $y_k\to x$
                and
                $\alpha_k\,(y_k-x_k) \to v$} \}
                .
        \end{align*}
                The sets
                $\Tan(\theset,x)$ and $\Ptg(\theset,x)$ are called the {\em contingent cone}
                (sometimes also the {\em tangent cone})
                of $\theset$ at $x$
                and the {\em paratingent cone}
                of $\theset$ at $x$,
                respectively.
                Let $\der \theset$ denote the set of accumulation points of $\theset$.

                We need the following generalization of \cite[Lemma~4.9]{KZ}
                to vector-valued functions:
        \begin{lemma}
                \label{l:KZ49}
                Let $Y$ be a~normed linear space.
                Let $\theset\subset \rn$,
                $x\in \theset \cap \der \theset$,
                $f\fcolon \theset\to Y$
                and let $L\in \mathcal L(\rn, Y)$ be a~strict derivative of $f$ at $x$.
                Let
                $v_1,\dots,v_n\in \Ptg(\theset,x)$ be unit vectors and
                $|\det(v_1,\dots,v_n)|>d>0$. Then
        \begin{equation}\label{eq:limsupKZ49}
                \left\|L\right\|_{\mathcal L(\rn, Y)}
                \le
                \limsup\limits_{\substack{y\to x,\ z\to x,\\ y,z\in \theset,\ y\neq z}}
                 \frac{n}{d}
                \frac{\left\|f(y) - f(z)\right\|_Y}{\left|y-z\right|}
                .
        \end{equation}
        \end{lemma}
        \begin{proof}
                For the case $Y=\R$, see \cite[Lemma 4.9]{KZ}
                where 
                \eqref{eq:limsupKZ49} is presented in the form
        \[
                \limsup\limits_{\substack{y\to x,\ z\to x,\\ y,z\in \theset,\ y\neq z}}
                \frac{\left\|f(y) - f(z)\right\|_Y}{\left|y-z\right|}
                > \frac{d\varepsilon}{n}
        \]
                whenever
                $\left\|L\right\|_{\mathcal L(\rn, Y)} > \varepsilon > 0$.
                We only need to treat the vector case.
                If $\left\|L\right\|_{\mathcal L(\rn, Y)}> \varepsilon > 0$ then there obviously exists $\phi \in Y^*$ such that
                $\left\| \phi \right\|_{Y^*} = 1$ and
                $\left\| \phi L \right \|_{\mathcal L(\rn, \R)} > \varepsilon$.
                Now it is enough to apply \cite[Lemma 4.9]{KZ} to the real function $f_0(y)=\phi(f(y))$, $y\in \theset$.
        \end{proof}

        From Lemma~\ref{l:KZ49}, the following statement inspired by
        \cite[Corollary~2]{ALP}
        (see also \cite[Lemma~U]{KZ})
        immediately follows:
        \begin{lemma}\label{l:U2}
                Let $Y$ be a~normed linear space,
                $\theset \subset \rn$
                and
                $x\in \theset \cap \der \theset$.
                If
                $f\fcolon \theset\to Y$ is
                relatively
                strictly differentiable at $x$
                and $\Ptg(\theset,x)$ spans $\rn$, then the
                relative
                strict derivative of $f$ at $x$ is determined uniquely.
        \end{lemma}
        \begin{proof}
                If $\Ptg(\theset,x)$ spans $\rn$ then $\Span \{v_1,\dots,v_n\} = \rn$ for some vectors $v_1,\dots,v_n\in \Ptg(\theset,x)$.
                We can assume that $v_1,\dots,v_n$ are unit vectors.
                Choose any $0 < d< |\det(v_1,\dots,v_n)|$. Assume that $f\fcolon \theset\to Y$
                is strictly differentiable at~$x$ and $L_1$, $L_2$ are two distinct strict derivatives of $f$ at $x$.
                Let $f_0(y) = f(y) - L_1(y)$ for $y\in \theset$ and $L_0 = L_2 - L_1$.
                Then $L_0\neq 0$ is a~strict derivative of $f_0$ at $x$.
                By Lemma~\ref{l:KZ49}, \eqref{eq:limsupKZ49} holds true
                for $f_0$ and $L_0$ with $\|L_0\|_{\mathcal L(\rn, Y)} > 0$.
                However,
                this
                contradicts
                the fact
                that $0$ is also a~strict derivative of $f_0$ at $x$.
        \end{proof}
        Now, we can formulate another corollary to Theorem~\ref{thm:fininf}.
        The following result is a~generalization of \cite[Proposition~4.10]{KZ}
          (we allow vector-valued mappings and also replace $\Tan(\theset,x)$ by a larger set $\Ptg(\theset,x)$).
        \begin{proposition}\label{prop:KZ410}
                Let $\theset\subset \rn$ be a~nonempty closed set
                such that\/
                $\Ptg(\theset,x)$
                spans $\rn$ for every $x\in \der \theset$.\footnote{Of course, $\Ptg(\theset,x)$ can be replaced by $\Tan(\theset,x)$, which is
                smaller
                but probably better known.}
                Let $Y$ be a~normed linear space
                and $f\fcolon \theset\to Y$ a~function (relatively) strictly differentiable at every $x\in
                \der \theset
                $.
                Then there exists a~differentiable extension of $f$ defined on $\rn$.
        \end{proposition}
        \begin{proof}
                Let
                $L\fcolon \theset \to \mathcal L(\rn, Y)$ be a~(relative) strict derivative of~$f$ on~$\theset$.
                Note that $L$ is uniquely determined on $\der \theset$ by
                Lemma~\ref{l:U2}.
                To finish the proof
                with the help of Theorem~\ref{thm:fininf},
                we only need
                to prove that $L$ is a~Baire one function on~$\theset$.
                We follow \cite{KZ} by letting
                \[
                        \theset_m = \left\{
                                         x \in \der \theset
                                       \setcolon
                                         \sup 
                                            \left\{
                                                  \det(v_1,\dots,v_n)
                                                 \setcolon
                                                  v_1,\dots,v_n \text{ unit vectors from }\Ptg(\theset,x)
                                            \right\}
                                          \ge \frac{1}{m}
                              \right\}
                        \qquad
                        \text{ for } m\in \N
                        .
                \]
                By the proof of \cite[Proposition~4.10]{KZ}, $\theset_m$ is closed for every $m\in \N$ and
                $\bigcup_{m\in \N} \theset_m = \der \theset$ (the fact that $\Ptg(\theset,x)$ spans $\rn$ for every $x\in \der \theset$ is used).
                Authors of \cite{KZ} also prove that $L$ is continuous on $\theset_m$ for every $m\in \N$, and the same proof can be applied unchanged
                to vector-valued functions
                $f$ and
                $L$, with
                the
                help of Lemma~\ref{l:KZ49} instead of \cite[Lemma~4.9]{KZ}.

%
%

                Now, it
                is easy to deduce
                that $L$ is $\Fsigma$-measurable on $\theset$ and \cite{KZ} deduce
                that $L$ is Baire one. For this step with vector-valued functions, we need
                to refer to \cite[Corollary 4.13]{Kar2015}.
                Alternatively, we
                construct a sequence of continuous functions $\{L_m\}_{m\in\N}$ that converges point-wise to $L$.
                Observe that $\theset_{m+1} \supset \theset_m$ for every $m\in \N$ and
                let $H_n$ ($n\in \N$) be an increasing sequence of finite sets such that
                $\theset\setminus \der \theset \subset \bigcup_{m\in \N} H_m \subset \theset$.
                By Dugundji's extension theorem
                or our
                Theorem~\ref{thm:difext}\itemref{thm:difext:item:cont}\itemref{thm:difext:item:contcomp},\footnote{
                        We use
                        Theorem~\ref{thm:difext}
                        with $f_{\text{Theorem\,\ref{thm:difext}}}:=L|_{F_m\cup H_m}$,
                        $L_{\text{Theorem\,\ref{thm:difext}}}:=0$,
                        $Y_{\text{Theorem\,\ref{thm:difext}}} := \mathcal L(\rn,Y)$
                        and then we let $L_m = ( \bar f_{\text{Theorem~\ref{thm:difext}}} )|_ F$.}
                we let $L_m \fcolon \theset \to \mathcal L(\rn,Y)$ 
                be a~continuous extension
                of the (continuous) function $L|_{\theset_m \cup H_m}$.
                Obviously, $L$ is the point-wise limit of $L_m$.
        \end{proof}

                Once we assume strict differentiability, is it natural to ask about stronger conclusions
                with strict differentiability or $C^1$ smoothness.
                Of course, we need an additional assumption
                that enforces continuity of the strict derivative.
        \begin{proposition}\label{prop:KZstrictstrict}
               Under the assumptions of
               Proposition~\ref{prop:KZ410},
               there exists a~differentiable extension $\bar f\fcolon \rn \to Y$
               of $f$
               such that
                $\bar f$ is
                strictly differentiable
                at $x$
                (with respect to~$\rn$)
                and the derivative of $\bar f$ is continuous
                at $x$
                (with respect to~$\rn$)
                for
                all
        \myParBeforeItems
        \begin{enumerate}[(a)]
                \item
                \label{item:NaDoplnku}
                $x\in \rn\setminus \der \theset$ and
                \item
                \label{item:strictlyconti}
                $x\in \der \theset$ where the (unique)
                relative
                strict
                derivative of $f$
                (with respect to
                $\theset$)
                is continuous
                with respect to
                $\der \theset$.
        \end{enumerate}
        \end{proposition}
        \begin{proof}
                Let $L_0\fcolon \theset \to \mathcal L(\rn, Y)$ be a~strict derivative of~$f$ on~$\theset$
                (recall that 
                at isolated points,
                any element of $\mathcal L(\rn, Y)$ is a~strict derivative of $f$ with respect to $\theset$).
                By the proof of Proposition~\ref{prop:KZ410}, $L_0$
                is a~Baire one function on~$\theset$.
                So is the restriction $L_0|_{\der \theset}$ to the closed set $\der \theset$.
                Using Theorem~\ref{thm:B1ext}
                with $L=L_0|_{\der \theset}$, we obtain $A\setcolon \rn \setminus \der \theset \to \mathcal L(\rn, Y)$
                such that the ``union of functions'' $L_1:= L_0|_{\der \theset} \cup A$
                is an extension of $L_0|_{\der \theset}$ that is
                Baire one on $\rn$ and 
                continuous at every point of $\rn \setminus \der \theset$
                and at every point of $\der \theset$ where $L_0|_{\der \theset}$ is continuous.
                Let
                $L=L_1|_\theset$.
                The extension $\bar f$ of~$f$ provided by Theorem~\ref{thm:fininf}
                (with special regard to \itemref{thm:fininf:item:strict})
                is strictly differentiable at~all points promised. 
        Let $x\in \theset$ be as
        in~\itemref{item:NaDoplnku}
        or~\itemref{item:strictlyconti}.
        Then
        the continuity of the derivative of~$\bar f$ at~$x$
        with respect to $(\rn\setminus \theset)\cup \{x\}$ comes from
        Theorem~\ref{thm:fininf}\itemref{thm:fininf:item:strict}.
        This combines with the
        continuity
        of the relative strict derivative of~$f$
        at~$x$ with respect to~$\theset$
        assumed in~\itemref{item:strictlyconti}.
        \end{proof}

        The continuity
        assumption
        cannot be removed
        from~Proposition~\ref{prop:KZstrictstrict}\itemref{item:strictlyconti}
        as can be seen from \cite[Example~4.14]{KZ}.
        However, it can be replaced by an assumption on $\theset$,
        see \cite{KolarClanekIII}.

        A particular corollary under the assumptions of Proposition~\ref{prop:KZstrictstrict}
        is
        the following statement:
        if the (unique)
                relative
                strict
        derivative of~$f$ on~$\der \theset$ is
        assumed to be continuous on the whole $\der \theset$
        then there exists a~$C^1$ extension of $f$ to $\rn$.
        However, this follows from Whitney's theorem, even without the assumptions on $\Ptg(\theset,x)$.
        (The extension formula \cite[(11.1)]{W} of Whitney works well for vector-valued functions.)

\end{appendices}

\end{document}